\documentclass[11pt]{amsart}
\usepackage{tikz-cd}
\usepackage{verbatim}
\usepackage{mathtools}
\usepackage[bookmarks,colorlinks,breaklinks]{hyperref}  
\hypersetup{linkcolor=blue,citecolor=blue,filecolor=blue,urlcolor=blue}

\newtheorem{theorem}{Theorem}[section]
\newtheorem{proposition}[theorem]{Proposition}
\newtheorem{lemma}[theorem]{Lemma}
\newtheorem{claim}[theorem]{Claim}
\newtheorem{fact}[theorem]{Fact}

\newtheorem{definition}[theorem]{Definition}

\newtheorem{conjecture}[theorem]{Conjecture}
\newtheorem{notation}[theorem]{Notation}
\theoremstyle{plain}
\numberwithin{equation}{theorem}

\theoremstyle{remark}
\newtheorem{remark}[theorem]{Remark}

\newcommand{\Fpbar}{\overline{\mathbb{F}_p}}
\newcommand{\N}{{\mathbb N}}
\newcommand{\End}{{\rm End}}
\newcommand{\isomto}{\overset{\sim}{\rightarrow}}
\newcommand{\Q}{{\mathbb Q}}

\newcommand{\C}{\mathbb C}
\newcommand{\bG}{{\mathbb G}}
\newcommand{\bP}{{\mathbb P}}
\newcommand{\Z}{{\mathbb Z}}

\newcommand{\matO}{{\mathcal O}}
\newcommand{\Fp}{\mathbb F_p}
\newcommand{\Fq}{\mathbb F_q}

\newcommand{\lra}{\longrightarrow}
\newcommand{\dra}{\dashrightarrow}

\newcommand{\tensor}{\otimes}
\newcommand{\OO}{\matO}
\newcommand{\id}{{\rm id}}
\newcommand{\Pl}{{\mathbb P}}
\newcommand{\A}{\mathbb A}
\newcommand{\Qbar}{\overline{\Q}}

\newcommand{\trdeg}{{\rm trdeg}}

\title[Zariski dense orbits]{Zariski dense orbits for regular self-maps of split semiabelian varieties  in positive characteristic}
\author{Dragos Ghioca and Sina Saleh}


\address{
Dragos Ghioca \\
Department of Mathematics\\
University of British Columbia\\
Vancouver, BC V6T 1Z2\\
Canada
}

\email{dghioca@math.ubc.ca}

\address{
Sina Saleh \\
Department of Mathematics\\
University of British Columbia\\
Vancouver, BC V6T 1Z2\\
Canada
}

\email{sinas@math.ubc.ca}

\begin{document}
\maketitle
\begin{abstract}
 We prove the Zariski dense orbit conjecture in positive characteristic for regular self-maps of split semiabelian varieties. 
\end{abstract}

%%%%%%%%%%%%%%%%%%%%%%%%%%%%%%%%%%%%%%%%%%%%%%%%%%%%%%%%%%%%%%%%%%%%%%%%%%%%%%%
%%%%%%%%%%%%%%%%%%%%%%%%%%%%%%%%%%%%%%%%%%%%%%%%%%%%%%%%%%%%%%%%%%%%%%%%%%%%%%%

\section{Introduction}
\subsection{Notation}
We let $\N_0:=\N\cup\{0\}$ denote the set of nonnegative integers. 
For any self-map $\Phi$ on a variety $X$ and for any integer $n\ge 0$, we let $\Phi^n$ be the $n$-th iterate of $\Phi$ (where $\Phi^0$ is the identity map $\id:=\id_X$, by definition). For a point $x\in X$ with the property that each point $\Phi^n(x)$ avoids the indeterminacy locus of $\Phi$, we denote by $\OO_\Phi(x)$ the orbit of $x$ under $\Phi$, i.e., the set of all $\Phi^n(x)$ for $n\ge 0$. We say that $x$ is preperiodic if its orbit $\OO_\Phi(x)$ is finite; furthermore, if $\Phi^n(x)=x$ for some positive integer $n$, then we say that $x$ is periodic.

%%%%%%%%%%%%%%%%%%%%%%%%%%%%%%%%%%%%%%%%%%%%%%%%%%%%%%%%%%%%%%%%%%%%%%%%%%%%%%%

\subsection{The classical Zariski dense orbit conjecture}

The following conjecture was motivated by a similar question raised by Zhang \cite{Zhang} and was formulated by Medvedev and Scanlon \cite{M-S} and by Amerik and Campana \cite{A-C}.

\begin{conjecture}
\label{conj:char0}
Let $X$ be a quasiprojective variety defined over an algebraically closed field $K$ of characteristic $0$ and let $\Phi:X\dra X$ be a dominant rational self-map. Then either there exists $\alpha\in X(K)$ whose orbit under $\Phi$ is well-defined and Zariski dense in $X$, or there exists a non-constant rational function $f:X\dra \Pl^1$ such that $f\circ \Phi=f$.
\end{conjecture}

There are several partial results known towards Conjecture~\ref{conj:char0} (see \cite{A-C, BGSZ, BGZ, G-H, G-Sina, G-Matt, G-S, G-X, X}).

%%%%%%%%%%%%%%%%%%%%%%%%%%%%%%%%%%%%%%%%%%%%%%%%%%%%%%%%%%%%%%%%%%%%%%%%%%%%%%%%

\subsection{The picture in positive characteristic}

If $K$ has characteristic $p>0$, then Conjecture~\ref{conj:char0} does not hold due to the presence of the Frobenius endomorphism (see \cite[Remark~1.2]{G-Sina-20}). In particular, if $K=\Fpbar$ (the algebraic closure of the finite field $\Fp$), then each orbit of a point $\alpha\in X(K)$ is finite under a rational self-map $\Phi:X\lra X$ defined over $K=\Fpbar$; furthermore, $\Phi$ does not have to preserve a non-constant rational function.  So, the authors proposed the following  conjecture as a variant of conjecture \ref{conj:char0} in positive characteristic (see also \cite{G-Sina-20}).
\begin{conjecture}
\label{conj:charp}
Let $K$ be an algebraically closed field of positive transcendence degree over $\Fpbar$, let $X$ be a quasiprojective variety defined over $K$, and let $\Phi:X\dra X$ be a dominant rational self-map defined over $K$ as well. Then at least one of the following three statements must hold:
\begin{itemize}
\item[(A)] There exists $\alpha\in X(K)$ whose orbit $\OO_\Phi(\alpha)$ is Zariski dense in $X$. 
\item[(B)] There exists a non-constant rational function $f:X\dra \bP^1$ such that $f\circ \Phi=f$. 
\item[(C)] There exist positive integers $m$ and $r$, there exists a variety $Y$ defined over a finite subfield $\Fq$ of $K$ such that  $\dim(Y)\ge \trdeg_{\Fpbar}K + 1$ and there exists a dominant rational map $\tau: X \dra Y$ such that
\[
\tau \circ \Phi^m = F^r \circ \tau,
\]
where $F$ is the Frobenius endomorphism of $Y$ corresponding to the field $\Fq$.
\end{itemize}
\end{conjecture}

Clearly (as observed also in \cite{BGZ}), conclusion~(A) would prevent conclusion~(B) in Conjecture~\ref{conj:charp}. Also, using an argument similar to the one employed in \cite[Remark~1.2]{G-Sina-20}, one sees that conclusion~(C) would also prevent conclusion~(B) to hold. On the other hand, conclusions~(A)~and~(C) are not mutually exclusive as one can easily see from the following endomorphism $\Phi:\bG_m^3\lra \bG_m^3$ defined over $\Fp(t)$ by the following rule: 
$$\Phi(x_1,x_2,x_3)=\left(x_1,x_2^p,x_3^p\right);$$
in this case, $\Phi$ leaves invariant the projection map $\pi_1:\bG_3\lra \bG_m$ on the first coordinate, while $\Phi$ induces the Frobenius endomorphism on the last two coordinates of $\bG_m^3$.

In \cite[Proposition~1.7]{X-2}, Xie proved Conjecture~\ref{conj:charp} when $\trdeg_{\Fpbar}K\ge \dim(X)$. In this case, the alternative~(C) from Conjecture~\ref{conj:charp} never occurs. In our paper we will deal with Conjecture~\ref{conj:charp} for regular self-maps of split semiabelian varieties $X$ defined over $\Fpbar$, while the field $K$ has arbitrary transcendence degree; in our setting, conclusion~(C) from Conjecture~\ref{conj:charp} occurs and furthermore, it constitutes the most delicate point for our proofs.

\subsection{Our results}

We prove our Conjecture~\ref{conj:charp} in the case of regular self-maps $\Phi$ of split semiabelian varieties $G$ defined over $\Fpbar$ (see Theorem~\ref{thm:main4 0}), i.e., $G$ is isogenous to a product of an abelian variety with a torus (or alternatively, $G$ is isogenous with a product of simple semiabelian varieties). 
The case of $G$ being isomorphic to a torus has already been proven by the authors in \cite[Theorem 1.5]{G-Sina-20}; however, the general case of split semiabelian varieties is more subtle than the case of tori.

We prove the following more precise version of Conjecture~\ref{conj:charp} for the case of regular self-maps of semiabelian varieties defined over a finite field. 
\begin{theorem}
\label{thm:main4 0}
Let $K$ be an algebraically closed field of characteristic $p$ such that $\trdeg_{\Fpbar}K\ge 1$ and let $G$ be a split semiabelian variety defined over $\Fpbar$. Let $\Phi:G \lra G$ be a dominant regular self-map defined over $K$. Then at least one of the following statements  must hold.
\begin{itemize}
\item[(A)] There exists $\alpha\in G(K)$ whose orbit under $\Phi$ is Zariski dense in $G$.
\item[(B)] There exists a non-constant rational function $f:G\dra \bP^1$ such that $f\circ \Phi=f$.
\item[(C)] There exist positive integers $m$ and $r$, a semiabelian variety $Y$ defined over a finite subfield $\Fq$ of $K$ of dimension at least equal to $\trdeg_{\Fpbar}K + 1$ and a dominant regular map $\tau:G \lra Y$ such that 
\begin{equation}
\label{eq:C}
\tau \circ \Phi^m = F^r\circ \tau,
\end{equation} 
where $F$ is the usual Frobenius endomorphism of $Y$ induced by the field automorphism $x\mapsto x^q$.
\end{itemize}
\end{theorem}

%%%%%%%%%%%%%%%%%%%%%%%%%%%%%%%%%%%%%%%%%%%%%%%%%%%%%%%%%%%%%%%%%%%%%%%%%%%%%%%

\subsection{Discussion of our proof}

Our proof of Theorem~\ref{thm:main4 0} follows the general strategy we employed in \cite{G-Sina-20} to treat the case of algebraic tori; however, there are significant complications due to the more complex structure of the endomorphism ring of a semiabelian variety compared with the power of the multiplicative group $\bG_m^N$. In  particular, Sections~\ref{sec:reduction}~and~\ref{sec: case 2} contain technical difficulties which are significantly more delicate than any of the arguments necessary for the case of  tori.

Since each regular self-map $\Phi$ of a semiabelian variety $G$ is a composition of a translation with a group endomorphism of $G$ (see \cite[Theorem 2]{Iitaka}), one needs to understand the arithmetic dynamics associated with a group endomorphism of $G$. Using the fact that $G$ is a split semiabelian variety allows us to understand better the  dynamics associated to  a group endomorphism of $G$; in particular, extending the current proof to the case of non-split semiabelian varieties would be significantly more difficult. This is not surprising since even the case of the Zariski dense orbit conjecture in characteristic $0$ was significantly more difficult for non-split semiabelian varieties as opposed to the case of abelian varieties (or tori); see the new technical ingredients one needed to introduce in \cite{G-Matt} to treat general semiabelian varieties compared to the case of abelian varieties treated in \cite{G-S}.

Furthermore, the case of semiabelian varieties \emph{not} defined over $\Fpbar$ is significantly more complicated. Indeed, in either characteristic ($0$ or $p$), in order to treat the Zariski dense orbit conjecture in the case of semiabelian varieties $G$ defined over an algebraically closed field $K$ endowed with some dominant regular self-map $\Phi$, one  considers a point $\alpha\in G(K)$ and then notes that its orbit $\OO_\Phi(\alpha)$ is contained in some finitely generated subgroup $\Gamma\subset G(K)$. If $\OO_\Phi(\alpha)$ is not Zariski dense, then its Zariski closure $V$ is a proper subvariety of $G$; a key step is exploiting the \emph{precise} structure of the intersection $V(K)\cap \Gamma$. So, it is essential for one to have a clear picture for the structure of the intersection between a proper subvariety of $G$ with a finitely generated subgroup of $G(K)$; this is something that Faltings \cite{Faltings} and Vojta \cite{Vojta} provided if $K$ has characteristic $0$, while in the case of prime characteristic $p$, Moosa and Scanlon \cite{Moosa-S} provide a precise description under the \emph{additional} assumption that $G$ is defined over $\Fpbar$. If $G$ is some arbitrary abelian variety defined over a field of characteristic $p$ (with nontrivial trace over $\Fpbar$), then Moosa and Scanlon \cite{Moosa-S-2} provide a more complicated description, which is not easy to exploit for our particular arithmetic dynamical question.

%%%%%%%%%%%%%%%%%%%%%%%%%%%%%%%%%%%%%%%%%%%%%%%%%%%%%%%%%%%%%%%%%%%%%%%%%%%%%%%

\subsection{Plan for our paper}

In Section~\ref{sec:reduction} we show that it suffices to prove Theorem~\ref{thm:main4 0} when $G$ is a product of simple semiabelian varieties:
\begin{equation}
\label{G reduced}
G=\prod_{i=1}^r C_i^{k_i},
\end{equation}
for some positive integers $k_i$, where the $C_i$'s are non-isogenous simple semiabelian varieties (see Section~\ref{sec:reduction}, especially Lemmas~\ref{lem:C2},~\ref{lem:C1}~and~\ref{lem:C3} and Theorem~\ref{thm:main4 0 2}). In particular, our reduction requires a very careful analysis of the dynamics not only of group endomorphisms of a split semiabelian variety, but also of finite-to-finite correspondences, as defined in Section~\ref{sec:correspondences}. This last complication was not encountered when one deals with the case of tori (see \cite{G-Sina-20}).

Then the rest of our proof  is dedicated to proving Theorem~\ref{thm:main4 0 2} in the case $G$ is of the form \eqref{G reduced} (i.e., $G$ is \emph{reduced}, according to Definition~\ref{def:reduced}). Since the group endomorphisms of a reduced semiabelian variety $G$ is isomorphic to  a product of matrix rings  $\prod_{i=1}^r M_{k_i,k_i}(E_i)$, where $E_i:=\End(C_i)$ is a subring of a (possibly) skew field, then in Section~\ref{sec:technical} we present several useful facts regarding skew fields and the rings $E_i$ above which appear as endomorphism rings for some simple semiabelian variety defined over $\Fpbar$. Also, in Section~\ref{sec:technical}, we present several other useful technical facts to be used later in our proofs, such as the $F$-structure theorem of Moosa-Scanlon for the intersection of a subvariety with a finitely generated group (see \cite{Moosa-S, Moosa-S-2} and also \cite{G-TAMS}).  In Section~\ref{sec:unipotent}, we prove Theorem~\ref{thm:main4 0 2} in the case $\Phi$ is a unipotent group endomorphism of $G$. In Section~\ref{sec: case 2} we prove Theorem~\ref{thm:main4 0 2} in the case when $\Phi$ is a group endomorphism whose eigenvalues for its induced action on each $C_i^{k_i}$ (see \eqref{G reduced}) are  powers of the  Frobenius elements from the endomorphism rings of each $C_i$. This last case is the instance when alternative~(C) from Conjecture~\ref{conj:charp} occurs and therefore, it requires a very careful analysis (significantly more in-depth than what was needed in the case $G$ was an algebraic torus).  Sections~\ref{sec:non-unipotent}~and~\ref{sec:split_case} are dedicated to proving a couple of mixed cases for Theorem~\ref{thm:main4 0 2}, which are technical ingredients for finishing the proof of our main result in Section~\ref{sec:conclusion}.

%%%%%%%%%%%%%%%%%%%%%%%%%%%%%%%%%%%%%%%%%%%%%%%%%%%%%%%%%%%%%%%%%%%%%%%%%%%%%%%
%%%%%%%%%%%%%%%%%%%%%%%%%%%%%%%%%%%%%%%%%%%%%%%%%%%%%%%%%%%%%%%%%%%%%%%%%%%%%%%

\section{Technical background}
\label{sec:technical}

In this Section we gather the various technical background results we need from the theory of matrices over skew fields (see Section~\ref{sec:skew}) and the theory of semiabelian varieties (see Sections~\ref{sec:semiabelian}~to~\ref{sec:semiabelian_end}).

%%%%%%%%%%%%%%%%%%%%%%%%%%%%%%%%%%%%%%%%%%%%%%%%%%%%%%%%%%%%%%%%%%%%%%%%%%%%%%%

\subsection{Matrices over skew fields}
\label{sec:skew}

%\begin{fact}[Jordan normal form theorem]
%\label{fact:jordanForm}
%Let $K$ be a skew field with perfect centre $k$ and $F$ a
%commutative field containing an algebraic closure of $k$ as well as an
%element $\lambda$ transcendental over $k$. Form the coproduct $K \ast_k F$ (see %chapter 5 of \cite{Cohn}) and take an
%extension $L$ which is matrix-homogeneous (see p. 234 of \cite{Cohn}) with %centre $k$. Then every matrix
%$A$ over $L$ has a conjugate of the form
%\begin{equation}
%J_{\alpha_1, r_1} \bigoplus \cdots \bigoplus J_{\alpha_n, r_n} \bigoplus %\lambda I_s    
%\end{equation}
%where $r + s = n$, the order of $A$, and $\alpha_1, \dots, \alpha_n \in %\overline{k}$. We refer to $\alpha_1, \dots, \alpha_n$ as the eigenvalues of %$A$.
%\end{fact}
%\begin{proof}
%For a proof see \cite[Theorem 8.3.6]{Cohn}.
%\end{proof}

\begin{fact}
\label{fact:jordan-normal-form}
Let $K$ be a skew field with centre $k$ and $A \in M_n(K)$ be a matrix with a minimal polynomial equal to $(x - \alpha)^r$ for some $\alpha \in k$ and $r \in \N$. Then, there exist an invertible matrix $P \in M_n(K)$ such that 
\[
P^{-1}AP = J_{\alpha, r_1} \bigoplus \cdots \bigoplus J_{\alpha, r_m},
\]
where $J_{\alpha,s}$ is the $s$-by-$s$ Jordan canonical matrix having unique eigenvalue $\alpha$ and its only nonzero entries away from the diagonal being the entries in positions $(i,i+1)$ (for $i=1,\dots, s-1$), which are all equal to $1$. 
\end{fact}
\begin{proof}
This is a consequence of the Jordan normal form theorem (see \cite[Theorem~8.3.6]{Cohn}, and the discussion on pages 382 and 383). 
\end{proof}

\begin{fact}
\label{fact:split}
Let $K$ be a skew field. Suppose that $A \in M_{n}(K)$ is a matrix with minimal polynomial $p(x) = p_1(x)p_2(x)$ over $k$ where $p_1, p_2 \in k[x]$ and $p_1$ and $p_2$ are coprime. There exists an invertible matrix $P \in M_{n}(K)$ such that $P^{-1}A P = A_1 \bigoplus A_2$ where the minimal polynomial of $A_1$ and $A_2$ over $k$ are $p_1$ and $p_2$, respectively. 
\end{fact}
\begin{proof}
The proof is identical as in the case when $K$ is a (commutative) field. 
\end{proof}

%%%%%%%%%%%%%%%%%%%%%%%%%%%%%%%%%%%%%%%%%%%%%%%%%%%%%%%%%%%%%%%%%%%%%%%%%%%%%%%%

\subsection{Semiabelian varieties}
\label{sec:semiabelian}

We recall that a semiabeian variety $G$ defined over an algebraically closed field $L$ is an algebraic group variety, which is an extension of an abelian variety $A$ by a torus $\bG_m^N$:
\begin{equation}
\label{eq:ses}
1\lra \bG_m^N\lra G\lra A\lra 1.
\end{equation}
We say that $G$ is \emph{split} if the short exact sequence of algebraic groups from \eqref{eq:ses} splits. In this case, $G$ is isogenous to a product of \emph{simple} semiabelian varieties (i.e., semiabelian varieties which contain no proper semiabelian varieties).  

As previously noted (see \cite{Iitaka}), any regular self-map on a semiabelian variety $G$ is a composition of a group endomorphism of $G$ with a translation map $\tau_{\beta}$ (where $\tau_\beta(x)=x+\beta$ for each $x\in G$).

\begin{definition}
\label{def:reduced}
We define a split semiabelian variety $G$ to be \textbf{reduced} if $G$ is isomorphic to
\begin{equation}
     \prod_{i = 1}^rC_i^{k_i},
\end{equation}
where $k_1, \dots, k_r \in \N$ and $C_1, \dots,C_r$ are simple semiabelian varieties that are pairwise non-isogenous.
\end{definition}

%%%%%%%%%%%%%%%%%%%%%%%%%%%%%%%%%%%%%%%%%%%%%%%%%%%%%%%%%%%%%%%%%%%%%%%%%%%%%%%

\subsection{Correspondences on semiabelian varieties}
\label{sec:correspondences}

For a semiabelian variety $G$, a \emph{correspondence} or \emph{finite-to-finite map} is a dominant map  $\varphi\in \frac{1}{m}\cdot \End(G)$ for some $m\in\N$. In other words, there exists a positive integer $m$ such that composing the multiplication-by-$m$ map $[m]_G$ on $G$ with $\varphi$ yields a well-defined,  dominant endomorphism of $G$. Clearly, for each point $\alpha\in G$, we have that $\varphi(\alpha)$ consists of at most $m^{2\dim(G)}$ points, which all differ by a torsion point of $G$ of order dividing $m$.

%%%%%%%%%%%%%%%%%%%%%%%%%%%%%%%%%%%%%%%%%%%%%%%%%%%%%%%%%%%%%%%%%%%%%%%%%%%%%%%

\subsection{Almost commutative diagrams}
\label{sec:almost}

We call an \emph{almost commutative diagram}, a diagram of the following form:
\begin{equation}
\label{diagram:altogether 0}
\begin{tikzcd}
G' \arrow[r, "\Psi"]  \arrow[d, "g"'] & G' \arrow[d, "g"] \\
G  \arrow[r, "\Phi"] & G
\end{tikzcd}    
\end{equation}
where $G$ and $G'$ are semiabelian varieties, $g:G'\lra G$ is an isogeny, $\Psi:G'\lra G'$ is a group endomorphism, while $\Phi:G\lra G$ is a correspondence such that there exists a positive integer $m_0$ for which
\begin{equation}
\label{eq:almost 0}
[m_0]_G\circ g\circ \Psi = [m_0]_G\circ \Phi\circ g.
\end{equation}
In particular, we have that   $[m_0]_G\circ \Phi$ is a well-defined, regular endomorphism of $G$. Furthermore,   letting  $\hat{g}:G\lra G'$ be an isogeny such that 
$$\hat{g}\circ g=[m]_{G'}\text{ and }g\circ \hat{g}=[m]_G,$$
for some positive integer $m$ divisible by $m_0$, we obtain:  
\begin{equation}
\label{eq:almost 1}
[m]_G\circ\Phi = g\circ \Psi\circ \hat{g}.
\end{equation} 
Furthermore, since $m_0$ divides $m$, we have that $[m]_G\circ\Phi$ from \eqref{eq:almost 1} is a well-defined endomorphism. In particular, we see that we can take $m_0:=m$ in equation~\eqref{eq:almost 0} and thus, $\Phi\in \frac{1}{m}\cdot \End(G)$ is defined through equation~\eqref{eq:almost 1}.

%%%%%%%%%%%%%%%%%%%%%%%%%%%%%%%%%%%%%%%%%%%%%%%%%%%%%%%%%%%%%%%%%%%%%%%%%%%%%%%

\subsection{Some technical definitions and notation relevant to our proofs}
\label{sec:semiabelian_end}

We will also employ the following notation regarding the endomorphisms of a (simple) semiabelian variety. 
\begin{notation}
\label{not:endomorphism}
In our paper, for a simple semiabelian variety $C$ defined over $\Fpbar$, there exists a finite subfield $\Fq$ of $\Fpbar$ such that:
\begin{itemize}
\item[(1)] $C$ is defined over $\Fq$; and
\item[(2)] each group endomorphism of $C$ is also defined over $\Fq$.
\end{itemize}
So, the (group) endomorphisms of $C$ defined over $\Fpbar$ belong to a ring $\End(C)=\End_{\Fpbar}(C)$. Following the notation of Milne \cite{Milne}, we let $\End^0:=\End(C)\otimes_{\Z}\Q$; then $L_C:=\End^0(C)$ is a skew field (all of whose elements are algebraic over $\Q$). We identify $\End(C)$ with a given subring $E_C$ of $L_C$; when there is no confusion, we will drop the index $C$ from our notation. 

Finally, we let $F$ be the Frobenius endomorphism of $C$ corresponding to the field $\Fq$ and we denote by $F_C$ its image in $E_C$.
\end{notation}

\begin{fact}
\label{fact:Frobenius}
With the same convention as in Notation~\ref{not:endomorphism} for $C$, $E_C$, $L_C$ and $F_C$, we have that $L_C$ is a skew field whose center $k_C$ contains $F_C$ (see \cite[Chapter~2]{Milne}). 
\end{fact}

Next definition is motivated by condition~(C) in Theorem~\ref{thm:main4 0}.
% \begin{definition}
% Let $C$ be a simple abelian variety defined over $\Fq$, along with $E_C=\End(C)$, $L_C=\End^0(C)$ and with the image $F_C\in E_C$ of the Frobenius endomorphism $F$ of $C$ over $\Fq$, as in Notation~\ref{not:endomorphism}. 
% Let $k\in\N$ and let $A \in M_{k, k}(E_C)$. For every positive integer $n$ we define $\JF_{m, n}(A)$ to be the maximum $\ell$ for which there exist $\ell$ Jordan blocks in the Jordan normal form of $A$ corresponding to eigenvalues $\lambda_1, \dots, \lambda_\ell$ (contained in the algebraic closure of $k_C$), where 
% \[
% \lambda_1^m = \cdots = \lambda_k^m = F_C^n.
% \]
% \end{definition}

\begin{definition}[NFP matrices]
\label{def:NFP}
Let $C$ be a simple  semiabelian variety along with the convention from Notation~\ref{not:endomorphism} for $E_C$, $L_C$ and $F_C$. For any $n \in \N$, a matrix $A \in M_{n, n}(L_C)$ is called an NFP (No Frobenius Power) matrix whenever the minimal polynomial $P(x)$ of $A$ over $\Q\left(F_{C}\right)$ has no roots that are multiplicatively dependent with respect to $F_{C}$ (i.e., no root $\lambda$ of $P(x)$ satisfies $\lambda^m=F_C^k$ for some integers $m$ and $k$, not both equal to $0$). 
\end{definition}

%%%%%%%%%%%%%%%%%%%%%%%%%%%%%%%%%%%%%%%%%%%%%%%%%%%%%%%%%%%%%%%%%%%%%%%%%%%%%%%

\subsection{The intersection of a subvariety of a semiabelian variety defined over a finite field with a finitely generated subgroup}
\label{sec:M-S}

We conclude this technical background section with stating the $F$-structure theorem of Moosa-Scanlon for the intersection of a subvariety of an $\Fq$-semiabelian variety $G$ with a finitely generated subgroup of $G(K)$ (where $K$ is some arbitrary algebraically closed field containing $\Fq$).  In order to state Theorem~\ref{Moosa-Scanlon theorem} (which is an essential ingredient for our proofs), we need first to introduce the notion of \emph{$F$-sets} defined by Moosa-Scanlon \cite{Moosa-S}. The Frobenius $F$ acting on $G$ is the endomorphism induced by the usual field homomorphism given by $x\mapsto x^q$ for each $x\in K$.

\begin{definition} 
\label{definition F-sets}
With the above notation for $G$, $q$, $K$ and $F$, let $\Gamma\subseteq G(K)$ be a finitely generated $\Z[F]$-module.
\begin{itemize}
\item[(a)]
By a \emph{sum of $F$-orbits in $\Gamma$} we mean a set of the form
$$C(\gamma,\alpha_1,\dots,\alpha_m;k_1,\dots,k_m):=\left\{\gamma+\sum_{j=1}^m F^{k_jn_j}(\alpha_j) \colon n_j\in\mathbb{N}_0\right\}\subseteq\Gamma$$
where $\gamma,\alpha_1,\dots,\alpha_m$ are some given points in $G(K)$ and $k_1,\dots,k_m$ are some given positive integers. 
\item[(b)]
An \emph{$F$-set} in $\Gamma$ is a set of the form 
$C+ \Gamma '$ where $C$ is a sum of $F$-orbits in $\Gamma$, and $\Gamma '\subseteq \Gamma$ is a subgroup, while in general, for two sets $A,B\subset G(K)$, $A+ B$ is simply the set $\{a+ b\colon a\in A\text{,  }b\in B\}$. 
\end{itemize}
\end{definition}

We note that since we allow the base points $a_i$ be outside $\Gamma$, we can use the slightly simpler definition of groupless $F$-sets involving sums of $F$-orbits rather than using the $F$-cycles (see \cite[Remark~2.6]{Moosa-S}, and also the  extension proven in \cite{G-TAMS}). We also refer to \cite[Section~2.2]{CGSZ} for a more in-depth discussion of the structure of $F$-sets.

\begin{theorem}[Moosa-Scanlon \cite{Moosa-S}]
\label{Moosa-Scanlon theorem}
Let $G$ be a semiabelian variety defined over $\Fq$, let $\Fq\subset K$ be an algebraically closed field, let $V\subset G$ be a subvariety defined over $K$ and let $\Gamma\subset G(K)$ be a finitely generated subgroup. Then $V(K)\cap \Gamma$ is a finite union of $F$-sets contained in $\Gamma$.
\end{theorem}

\begin{remark}
\label{rem:M-S}
Furthermore, according to \cite[Remark~2.6]{Moosa-S}, if $\Gamma$ is a finitely generated $\Z[F]$-submodule of $G(K)$, then the $F$-sets appearing in the intersection $V\cap\Gamma$ from Theorem~\ref{Moosa-Scanlon theorem} are of the form $C(\gamma,\alpha_1,\dots,\alpha_m;k_1,\dots,k_m)+\Gamma '$ (see Definition~\ref{definition F-sets}) where for some positive integer $\ell$, we have that 
$$\ell\cdot \gamma,\ell\cdot\alpha_1,\dots, \ell\cdot \alpha_m\in \Gamma.$$ 
Finally, in our proof, we prefer to use the notation
$$\Sigma(\alpha_1,\dots,\alpha_m;k_1,\dots,k_m):=\left\{\sum_{j=1}^m F^{k_jn_j}(\alpha_j) \colon n_j\in\mathbb{N}_0\right\}$$
for a sum of $F$-orbits (for given points $\alpha_j\in G(K)$ and positive integers $k_j$).
\end{remark}

%%%%%%%%%%%%%%%%%%%%%%%%%%%%%%%%%%%%%%%%%%%%%%%%%%%%%%%%%%%%%%%%%%%%%%%%%%%%%%%
%%%%%%%%%%%%%%%%%%%%%%%%%%%%%%%%%%%%%%%%%%%%%%%%%%%%%%%%%%%%%%%%%%%%%%%%%%%%%%%

\section{Reducing Theorem \ref{thm:main4 0} to the case of reduced split semiabelian varieties} 
\label{sec:reduction}

In this Section we show that it suffices to prove Theorem~\ref{thm:main4 0} when $G$ is a reduced semiabelian variety (see Theorem~\ref{thm:main4 0 2}). We start by recalling the setup from Theorem~\ref{thm:main4 0}. We have an algebraically closed field $K$ of positive transcendence degree over $\Fpbar$ and we have a split semiabelian variety $G$ defined over $\Fpbar$. 

Let $\Psi:G\lra G$ be a dominant regular self-map. Then $\Psi:=\tau_{\beta} \circ \psi$, where $\psi: G \lra G$ is a dominant group endomorphism and $\tau_{\beta}:G\lra G$ is the translation-by-$\beta$ map on $G$ (for some point $\beta\in G(K)$). Then for each $n\in\N$, we have that
\begin{equation}
\label{eq:formula iterate}
\Psi^n=\tau_{\sum_{j=0}^{n-1}\psi^j(\beta)}\circ \psi^n.
\end{equation}

The group endomorphism $\psi$ is integral over $\Z$ (see \cite[Section~2.1]{CGSZ}); so, we denote by $g_\psi$ the minimal monic polynomial with integer coefficients for which $g_\psi(\psi)=0$. Since $\psi$ is dominant, then each root of $g_\psi$ is nonzero. 

%%%%%%%%%%%%%%%%%%%%%%%%%%%%%%%%%%%%%%%%%%%%%%%%%%%%%%%%%%%%%%%%%%%%%%%%%%%%%%%

\subsection{Reduction to the case the roots of $g_\psi$ are not roots of unity of order greater than $1$}

We first note the following reduction in Theorem~\ref{thm:main4 0}. 

\begin{proposition}
\label{prop:iterate reduction}
In order to prove Theorem \ref{thm:main4 0} for the dynamical system $(G,\Psi)$, it suffices to prove Theorem~\ref{thm:main4 0} for the dynamical system $(G,\Psi^n)$ for some $n\in\N$. 
\end{proposition}

\begin{proof}
It is clear that if condition~(C) holds for an iterate of $\Psi$ then it also holds for $\Psi$. The fact that if conditions~(A)~and~(B) hold for an iterate of $\Psi$ then they also hold for $\Psi$ follows from \cite[Lemma~2.1]{BGSZ}. 
\end{proof}

After replacing $\Psi$ by a suitable iterate (see Proposition~\ref{prop:iterate reduction} and also formula~\ref{eq:formula iterate}) we may assume without loss of generality that the roots of the minimal polynomial of $\psi$ (over $\Z$) that are roots of unity are actually  all equal to one. 

%%%%%%%%%%%%%%%%%%%%%%%%%%%%%%%%%%%%%%%%%%%%%%%%%%%%%%%%%%%%%%%%%%%%%%%%%%%%%%%

\subsection{Writing the minimal polynomial of $\psi$ as a product of two coprime polynomials with special properties}
\label{subsec:previous}

Let $g:=g_\psi\in \Z[x]$ be the minimal polynomial for the endomorphism $\psi$. As explained in the previous section, we may assume that each root of $g$ is either equal to $1$ or not a root of unity.

We let $s\in\N_0$ be the order of $1$ as a root of $g(x)$. We write 
$h_1(x):=(x-1)^s$; then we can write $g(x) := h_1(x) \cdot h_2(x)$ for some  polynomial $h_2(x)$ with integer coefficients whose roots are not roots of unity. Furthermore, 
$h_1(x)$ and $h_2(x)$ are coprime polynomials. 

%%%%%%%%%%%%%%%%%%%%%%%%%%%%%%%%%%%%%%%%%%%%%%%%%%%%%%%%%%%%%%%%%%%%%%%%%%%%%%%

\subsection{Splitting the action of $\Psi$ to an action on a product of two special semiabelian varieties}

We continue with the notation for $h_1(x)$ and $h_2(x)$ from Section~\ref{subsec:previous} and we let $G_1 := h_2(\psi)\big(G\big)$ and $G_2 := h_1(\psi)\big(G\big)$. Then $G_1$
and $G_2$ are both connected algebraic subgroups of $G$ (note that either $G_1$ or $G_2$ may be the trivial group). Since $h_1$ and $h_2$ are coprime, then there exist polynomials with integer
coefficients $Q_1(x)$ and $Q_2(x)$ along with some positive integer $\ell_0$ such that
\[
Q_1(x)\cdot h_1(x) + Q_2(x)\cdot h_2(x) = \ell_0,
\]
which means that $G_1$ and $G_2$ are complementary subgroups of $ G$, in the sense that $ G = G_1 + G_2$, while $G_1\cap G_2$ is finite (consisting only
of points of order dividing $\ell_0$). Thus, for each $x\in G$ one can find $x_1 \in G_1$ and $x_2\in G_2$ such that $x=x_1+x_2$; even though $x_1$ and $x_2$ are not uniquely defined by $x$, since $G_1\cap G_2$ consists only of points of order dividing $\ell_0$, we conclude that the isogeny $\iota:G\lra G_1\times G_2$ given by 
\begin{equation}
\label{eq:well-defined}
x\mapsto (\ell_0x_1,\ell_0x_2)\text{ is well-defined.}
\end{equation}

Furthermore, $\psi$ induces endomorphisms of both $G_1$ and $G_2$; call them
$\psi_1$, respectively $\psi_2$. In addition, 
\begin{equation}
\label{eq:poly_1}
\text{the minimal polynomial of $\psi_1$
is $h_1(x)=(x-1)^s$,} 
\end{equation}
\begin{equation}
\label{eq:poly_2}
\text{while the minimal polynomial of $\psi_2$ is $h_2(x)$.} 
\end{equation}

Since $G_1+G_2=G$, then there exist $\beta_1\in G_1(K)$ and $\beta_2\in G_2(K)$ such that $\beta_1+\beta_2=\beta$. Furthermore, according to \eqref{eq:well-defined}, regardless of our choice of $(\beta_1,\beta_2)\in G_1\times G_2$ for which $\beta_1+\beta_2=\beta$, we have that the pair $(\ell_0\beta_1,\ell_0\beta_2)$ is unchanged. 

Now, we define $\Psi_1:G_1\lra G_1$ and $\Psi_2:G_2\lra G_2$ given by
\begin{equation}
\label{eq:Psi_i}
\Psi_i(x)=\psi_i(x)+\ell_0\beta_i\text{ for }i=1,2.
\end{equation}
Then, using the isogeny $\iota$ (see \eqref{eq:well-defined}) along with the definition of $\Psi_1$ and $\Psi_2$ (see \eqref{eq:Psi_i}), we have that the following diagram commutes
\begin{equation}
\label{diagram:G1timesG2}
\begin{tikzcd}
G \arrow[r, "\Psi"] \arrow[d, "\iota"] & G \arrow[d, "\iota"]\\
G_1 \times G_2 \arrow[r, "{(\Psi_1, \Psi_2)}"] & G_1 \times G_2.
\end{tikzcd}
\end{equation}

%%%%%%%%%%%%%%%%%%%%%%%%%%%%%%%%%%%%%%%%%%%%%%%%%%%%%%%%%%%%%%%%%%%%%%%%%%%%%%%

\subsection{Reduction of the action of $\psi_1$ to an endomorphism of a reduced  split semiabelian variety}

Since $G_1$ is a semiabelian subvariety of $G$, then also $G_1$ is a split semiabelian variety and must be isogenous to a semiabelian variety
\begin{equation}
G'_1 := \prod_{i=1}^r C_i^{k_i},    
\end{equation}
where the $C_i$'s are non-isogenous simple semiabelian varieties. More
precisely, we have an isogeny
\begin{equation}
\pi : G_1 \lra G'_1    
\end{equation}
and another isogeny
\begin{equation}
\hat{\pi} : G'_1\lra G_1    
\end{equation}
along with some positive integer $m_1$ such that
\begin{equation}
\label{eq:pi}
\pi\circ \hat{\pi} = [m_1]_{G'_1}, \quad \hat{\pi} \circ \pi = [m_1]_{G_1}. 
\end{equation}

Consider $\varphi_1' \in \End(G'_1) \otimes \Q$ given by
\begin{equation}
\label{eq:pi_2}
\varphi_1' := \frac{1}{m_1}\pi\circ \psi_1\circ \hat{\pi}.
\end{equation}
Then $\varphi_1'$ corresponds to a direct sum $A_1 \oplus \cdots \oplus A_r$ of matrices in
$$\prod_{i=1}^r M_{k_i}\left(\End(C_i)^0\right).$$  
Furthermore, since the minimal polynomial of each $A_i$ over $\Z$ is of the form $(x-1)^{s_i}$ for some integer $s_i\le s$ (see \eqref{eq:poly_1}), then each $A_i$ is a unipotent matrix. So, using Fact~\ref{fact:jordan-normal-form}, there exist matrices $P_i \in M_{k_i}(\End(C_i)^0)$ such that $P_iA_i P_i^{-1}$ is of the form 
\[
B_i = \bigoplus_{j=1}^{\ell_i} J_{1, m_j^{(i)}}
\]
for some $\ell_i\in\N$ and some positive integers $m_j^{(i)}$ such that 
$$\sum_{j=1}^{\ell_i}m_j^{(i)}=k_i.$$
This means (see Section~\ref{sec:almost}) that there must exist $\sigma, \hat{\sigma}: G'_1 \lra G'_1$ such that 
\begin{equation}
\label{eq:sigma}
\sigma \circ \hat{\sigma} = \hat{\sigma} \circ \sigma = [m_1']_{G'_1}
\end{equation}
for some $m_1' \in \N$ and  
\begin{equation}
\label{eq:sigma_2}
\varphi_1 = \frac{1}{m_1'}\sigma \circ \varphi_1' \circ \hat{\sigma},
\end{equation}
where $\varphi_1$ is the endomorphism corresponding to $B_1\oplus \cdots \oplus B_r$. 

%%%%%%%%%%%%%%%%%%%%%%%%%%%%%%%%%%%%%%%%%%%%%%%%%%%%%%%%%%%%%%%%%%%%%%%%%%%%%%%

\subsection{Reducing to the case $\Psi_2$ is a group endomorphism}

Now, since the minimal polynomial of $\psi_2$ (which is $h_2(x)$, according to \eqref{eq:poly_2}) does not have any roots that are equal to one, then we have that $\psi_2-{\rm id}_{G_2}$ is an dominant group endomorphism of $G_2$ and therefore, we can find $z\in G_2(K)$ such that 
\begin{equation}
\label{eq:kill beta_2}
\left(\psi_2-{\rm id}_{G_2}\right)(z)=\ell_0\beta_2.
\end{equation}
So, letting $\tau:G_2\lra G_2$ be the translation-by-$z$ map, then \eqref{eq:kill beta_2} yields that
\begin{equation}
\label{eq:psi_2 endomorphism}
\psi_2 :=\tau\circ \Psi_2\circ \tau^{-1}\text{ is an endomorphism of }G_2.
\end{equation}

%%%%%%%%%%%%%%%%%%%%%%%%%%%%%%%%%%%%%%%%%%%%%%%%%%%%%%%%%%%%%%%%%%%%%%%%%%%%%%%

\subsection{Reducing the dynamical system on $G$ to a simpler dynamical system on $G_1'\times G_2$}

We let $\nu:=\sigma\circ \pi$ and using \eqref{eq:pi}, \eqref{eq:pi_2}, \eqref{eq:sigma} and \eqref{eq:sigma_2}, we get that 
\begin{equation}
\label{eq:nu}
\nu\circ \psi_1=\varphi_1\circ \nu.
\end{equation}
We let $\Phi_1:G_1'\lra G_1'$ given by $\Phi_1(x)=\varphi_1(x)+\nu(\ell_0\beta_1)$. Since $\Psi_1:G_1\lra G_1$ is given by $\Psi_1(x)=\psi_1(x)+\ell_0\beta_1$ (see \eqref{eq:Psi_i}), then we conclude that
\begin{equation}
\label{eq:nu_2}
\nu\circ \Psi_1=\Phi_1\circ \nu.
\end{equation}
So, letting $g_1 := (\nu, \tau)$ and also using \eqref{eq:psi_2 endomorphism} and \eqref{eq:nu_2}, then we get the next commutative diagram 
\begin{equation}
\label{diagram:G'1timesG2}
\begin{tikzcd}
G_1 \times G_2 \arrow[r, "{(\Psi_1, \Psi_2)}"] \arrow[d, "g_1"] & G_1 \times G_2 \arrow[d, "g_1"]\\
G'_1 \times G_2 \arrow[r, "{(\Phi_1, \psi_2)}"] & G'_1 \times G_2.
\end{tikzcd} 
\end{equation}

%%%%%%%%%%%%%%%%%%%%%%%%%%%%%%%%%%%%%%%%%%%%%%%%%%%%%%%%%%%%%%%%%%%%%%%%%%%%%%%%

\subsection{Deconstructing the action of $\psi_2$ on $G_2$ using correspondences}
   
Since $G_2$ is a semiabelian subvariety of a split semiabelian variety, then also $G_2$ is isogenous to a reduced split semiabelian variety 
\begin{equation}
G'_2 := \prod_{i=1}^{r'} \left(C'_i\right)^{k'_i}.    
\end{equation}
Moreover, one can choose the components $C'_i$ so that for any $1 \le i \le r'$ and $1 \le j \le r$, $C'_i$ is isogenous to $C_j$ if and only if $C'_i = C_j$. In other words, the simple semiabelian components of $G'_1$ and $G'_2$ are either equal or they are non-isogenous. So, we have isogenies
\begin{equation}
\pi' : G_2 \lra G'_2, \quad \hat{\pi}' : G'_2\lra G_2 
\end{equation}
along with some positive integer $n_1$ such that
\begin{equation}
\label{eq:pi'}
\pi'\circ \hat{\pi}' = [n_1]_{G'_2}, \quad \hat{\pi}' \circ \pi' = [n_1]_{G_2}. 
\end{equation}

Now, consider $\varphi'_2\in \End(G_2')\otimes \Q$ given by
\begin{equation}
\label{eq:pi'_2}
\varphi'_2 := \frac{1}{n_1}\pi'\circ (\psi_2)\circ \hat{\pi}'. 
\end{equation}
Then  $\varphi'_2$ is a finite-to-finite map (or correspondence), i.e., it sends any finite subset
of $G'_2$ into another finite subset of $G'_2$. Also, we see that $\varphi'_2$ can be represented naturally in
\begin{equation}
\prod_{i=1}^{r'} M_{k'_i, k'_i}\left(\frac{1}{n_1} D_i\right),  
\end{equation}
where $D_i := \End(C'_i)$, while $\frac{1}{n_1} D_i$ means that we allow denominator $n_1$ for each entry in the corresponding matrices. We also fix an embedding of each $\Q(F_{C'_i})$ (for $i=1,\dots, r$) into $\Qbar$.

%%%%%%%%%%%%%%%%%%%%%%%%%%%%%%%%%%%%%%%%%%%%%%%%%%%%%%%%%%%%%%%%%%%%%%%%%%%%%%%%

\subsection{Linearizing the action of $\varphi_2'$ on $G_2'$}

The action of $\varphi'_2\in\End(G_2')\otimes \Q$ corresponds to a direct sum of matrices 
$$\tilde{A}_{\varphi_2'}:=A'_1 \oplus \cdots \oplus A'_{r'} \in \prod_{i=1}^{r'} M_{k'_i, k'_i}\left(\frac{1}{n_1} D_i\right).$$

%%%%%%%%%%%%%%%%%%%%%%%%%%%%%%%%%%%%%%%%%%%%%%%%%%%%%%%%%%%%%%%%%%%%%%%%%%%%%%%

\subsection{The minimal polynomial of $\tilde{A}_{\varphi_2'}$}

Using equations \eqref{eq:pi'} and \eqref{eq:pi'_2}, we see that for each $n\in\N$, we have that
\begin{equation}
\label{eq:same_denominator}
(\varphi'_2)^n := \frac{1}{n_1}\pi'\circ (\psi_2)^n\circ \hat{\pi}'
\end{equation}
and therefore, the minimal polynomial for $\tilde{A}_{\varphi_2'}$ (which is the matrix in $\End(G_2')\otimes\Q$ corresponding to $\varphi_2'$) is the same as the minimal polynomial of $\psi_2$ as an endomorphism of $G'_2$. Furthermore, using \eqref{eq:psi_2 endomorphism} along with \eqref{eq:poly_2}, we conclude that the minimal polynomial for $\tilde{A}_{\varphi_2'}$ is $h_2(x)$.

\begin{remark}
\label{rem:same_denominator}
Equation~\eqref{eq:same_denominator} also yields that for each $x\in G'_2$, we have that for any $n\in\N$ and for any two points $y,z\in (\varphi_2')^n(x)$ (i.e., for any two points $y$ and $z$ associated to $x$ by the correspondence $(\varphi_2')^n$), we have that
\begin{equation}
\label{eq:same_difference}
y-z\in G'_2[n_1]\text{ (i.e., it is a torsion point of order dividing $n_1$).}
\end{equation}
\end{remark}

%%%%%%%%%%%%%%%%%%%%%%%%%%%%%%%%%%%%%%%%%%%%%%%%%%%%%%%%%%%%%%%%%%%%%%%%%%%%%%%

\subsection{Separating the roots of $h_2(x)$}

As shown in the previous Section, we know that the minimal polynomial of the matrix $\tilde{A}_{\varphi_2'}$ is $h_2(x)$ and since $h_2(x)\in\Z[x]$ is a monic polynomial, we conclude that 
\begin{equation}
\label{eq:integral roots}
\text{each root of $h_2(x)$ is integral over $\Z$.}
\end{equation}

Using Proposition~\ref{prop:iterate reduction}, we can replace $\Psi$ by a suitable iterate (which leads to replacing $\varphi_2'$ by a corresponding iterate and therefore, replacing each matrix $A_i'$ by its suitable power), so that we may assume that the roots of the minimal polynomial of each $A'_i$ over $\Q\left(F_{C'_i}\right)$ are either a power of $F_{C'_i}$ or multiplicatively independent with respect to $F_{C'_i}$. Furthermore, writing each such multiplicatively dependent root of the minimal polynomial of $A_i'$ as $F_{C'_i}^{n_j^{(i)}}$ for some integer $n_j^{(i)}$ (where $1\le j\le s_i$ for some $s_i\in\N_0$), we note that the exponents $n_j^{(i)}$ must be positive integers because we know the roots of $h_2(x)$ are not roots of unity and also, we know that these roots must be integral over $\Z$, according to \eqref{eq:integral roots}.

%%%%%%%%%%%%%%%%%%%%%%%%%%%%%%%%%%%%%%%%%%%%%%%%%%%%%%%%%%%%%%%%%%%%%%%%%%%%%%%

\subsection{Splitting the action of $\tilde{A}_{\varphi_2'}$ into a suitable direct product}

Using facts \ref{fact:jordan-normal-form} and \ref{fact:split} along with the notation from the previous section regarding the roots of each minimal polynomial of $A_i'$ as being either of the form $F_{C'_i}^{n_j^{(i)}}$ for some $n_j^{(i)}\in\N$ (where $1\le j\le s_i$) or being multiplicatively independent with respect to $F_{C'_i}$, there must exist matrices $P'_i \in M_{k'_i}\left(\End(C'_i)^0\right)$ such that 
\begin{equation}
\label{eq:P_i'}
P'_iA'_i \left(P'_i\right)^{-1} = B_{1,i} \bigoplus B_{2,i},
\end{equation} 
where each $B_{1, i}$ is a Jordan matrix of the form 
\[
B_{1,i}:=\bigoplus_{j=1}^{s_i} J_{F_{C'_i}^{n_j^{(i)}}, \ell_j^{(i)}},
\]
where the $\ell_j^{(i)}$'s are positive integers and the entries of each $B_{2,i}$ lie inside $\frac{1}{\ell_2}D_i$ for some $\ell_2 \in \N$. At the expense of replacing $\ell_2$ by a suitable multiple, we may also assume that the entries of each $P_i'$ belong also to $\frac{1}{\ell_2}D_i$. Moreover, for every $i=1, \dots,r'$ the minimal polynomial of $B_{2,i}$ over $\Q(F_{C'_i})$ has no roots that are multiplicatively dependent with respect to $F_{C'_i}$. 

%%%%%%%%%%%%%%%%%%%%%%%%%%%%%%%%%%%%%%%%%%%%%%%%%%%%%%%%%%%%%%%%%%%%%%%%%%%%%%%

\subsection{From linear maps to endomorphisms and finite-to-finite maps}

Using the block decomposition given by \eqref{eq:P_i'}, there exists a  natural rearrangement of the simple components of $G'_2$ such that:
\begin{itemize}
\item $p:G_2'\isomto G'_3 \times G'_4$ is the isomorphism corresponding to this rearrangement of the simple components of $G_2'$;
\item there exists an endomorphism $\varphi_2$ of $G'_3$ corresponding to $\bigoplus_{i=1}^{r'} B_{1,i}$; and 
\item there exists a finite-to-finite map $\varphi_3$ on $G'_4$ corresponding to $\bigoplus_{i=1}^{r'} B_{2,i}$.
\end{itemize}
Let $\lambda: G'_2 \lra G'_2$ be the endomorphism corresponding to $\left(\ell_2P'_1\right) \oplus \cdots \oplus \left(\ell_2P'_r\right)$. If we let $g_2 := \left({\rm id}_{G'_1}, p \circ \lambda \circ \pi'\right)$, then we obtain the following diagram
\begin{equation}
\label{diagram:G'1G'2G'3}
\begin{tikzcd}
G'_1 \times G_2 \arrow[r, "{(\Phi_1, \psi_2)}"]  \arrow[d, "g_2"'] &[2em] G'_1 \times G_2 \arrow[d, "g_2"] \\
G'_1 \times G'_3 \times G'_4 \arrow[r, "\text{($\Phi_1, \varphi_2, \varphi_3$)}"] & G'_1 \times G'_3 \times G'_4. 
\end{tikzcd}    
\end{equation}

Combining \eqref{diagram:G'1G'2G'3} with \eqref{diagram:G'1timesG2} and \eqref{diagram:G1timesG2},  then we get the following  diagram 
\begin{equation}
\label{diagram:altogether}
\begin{tikzcd}
G \arrow[r, "\Psi"]  \arrow[d, "h"'] & G \arrow[d, "h"] \\
G'_1 \times G'_3 \times G'_4 \arrow[r, "\text{$(\Phi_1, \varphi_2, \varphi_3)$}"] & G'_1 \times G'_3 \times G'_4.
\end{tikzcd}    
\end{equation}

We note that neither \eqref{diagram:altogether} nor \eqref{diagram:G'1G'2G'3} are commutative diagrams since in both cases, the bottom map is only a correspondence (i.e., a finite-to-finite map). On the other hand, both those diagrams are \emph{almost} commutative, as we will explain next (we refer next to diagram~\eqref{diagram:altogether}, but the same argument applies also to diagram~\eqref{diagram:G'1G'2G'3}). So, we let 
$$G':=G'_1\times G'_3\times G'_4\text{ and also, let }\Phi:= (\Phi_1, \varphi_2, \varphi_3)$$
and note that there exists $\ell_2\in\N$ such that $[\ell_2]\circ \Phi$ is a well-defined regular morphism of $G'$. Thus, due to our definition of the maps from the diagram~\eqref{diagram:altogether}, we get that for each $x\in G$, we have that
\begin{equation}
\label{eq:almost_commuting}
(h\circ\Psi)(x)- (\Phi\circ h)(x)\in G'[\ell_2],
\end{equation}
since for any point $y\in G'$, we have that $\Phi(y)$ consists of finitely many points of the form $z+\xi$, for some $z\in G'$ and $\xi\in G'[\ell_2]$.
\begin{remark}
\label{rem:correspondence}
In terms of notation,  in \eqref{eq:almost_commuting} and also later on, for a point $y\in G'$, we let $\Phi(y)$ be \emph{any} of the finitely many points corresponding to $y$ in the finite-to-finite map $\Phi$. As previously noted, any two points in $\Phi(y)$ differ by an $\ell_2$-th torsion point of $G'$. 
\end{remark}

Furthermore, in light of Remark~\ref{rem:same_denominator}, a bit more is true: for any $n\in\N$, we have that $[\ell_2]\circ \Phi^n$ is a regular self-map on $G'$ and so, for each $x\in G$, we have that (see also the convention from Remark~\ref{rem:correspondence}) 
\begin{equation}
\label{eq:almost_commuting_2}
\left(h\circ\Psi^n\right)(x)- \left(\Phi^n\circ h\right)(x)\in G'[\ell_2].
\end{equation}

%%%%%%%%%%%%%%%%%%%%%%%%%%%%%%%%%%%%%%%%%%%%%%%%%%%%%%%%%%%%%%%%%%%%%%%%%%%%%%%

\subsection{The dynamics of finite-to-finite maps}

Our goal is to show that in order to prove Theorem~\ref{thm:main4 0} for the dynamical system $(G,\Psi)$, it suffices to prove Theorem~\ref{thm:main4 0} for the dynamical system given by the action of the finite-to-finite map $\Phi$ on  $G'=G'_1 \times G'_3 \times G'_4$. In order to show this, we first present some general facts regarding the dynamics of the finite-to-finite map $\Phi:G'\lra G'$.

\begin{definition}
\label{def:orbit}
Let $G'$ be a semiabelian variety and let $\Phi:G' \lra G'$ be a finite-to-finite map, i.e., a map of the form $\Phi:=\tau_\gamma \circ \varphi$, where $\tau_\gamma$ is the translation-by-$\gamma$ map on $G'$ (for a given point $\gamma\in G'$) and $\varphi\in \End(G)\otimes \Q$ (which means that there exists $\ell_1\in\N$ such that $[\ell_1]\circ \Phi$ is a well-defined regular self-map on $G'$). 

Let $x\in G'$; we say that the sequence of points $\{x_n\}_{n\ge 0}\subset G'$ is \emph{an orbit} of $x$ under $\Phi$ if $x_0=x$ and for each $n\ge 0$, we have that $x_{n+1}\in \Phi(x_n)$ (note that $\Phi(x_n)$ consists of finitely many points of $G$, which differ only by a torsion point of order dividing $\ell_1$).   
\end{definition}

We recall the \emph{almost} commuting diagram~\eqref{diagram:altogether}:
\begin{equation}
\label{diagram:altogether_2}
\begin{tikzcd}
G \arrow[r, "\Psi"] \arrow[d, "g"] & G \arrow[d, "g"] \\
G' \arrow[r, "\Phi"]& G'
\end{tikzcd}
\end{equation}
in which case we have that there exists some positive integer $\ell_2$ such that for each $x\in G$, we have (see also Remark~\ref{rem:correspondence}) 
\begin{equation}
\label{eq:almost_2}
\left(g\circ \Psi - \Phi\circ g\right)(x)\in G'[\ell_2].
\end{equation}

Also, very important for our setting is the fact that $G'=G'_1\times G'_3\times G'_4$ and that also $\Phi$ is a split map, i.e., $\Phi=(\Phi_1,\varphi_2,\varphi_3)$ in which $\Phi_1$ is a regular self-map of the semiabelian variety $G_1'$, and $\varphi_2$ is a group endomorphism of $G'_3$, while $\varphi_3$ is a finite-to-finite map on $G_4'$. 

In the next Sections we prove that each one of the conclusions~(A)-(C) from Theorem~\ref{thm:main4 0} can be inferred to $(G,\Psi)$ once they are known for $(G',\Phi)$.

%%%%%%%%%%%%%%%%%%%%%%%%%%%%%%%%%%%%%%%%%%%%%%%%%%%%%%%%%%%%%%%%%%%%%%%%%%%%%%%

\subsection{Condition (A) from Theorem~\ref{thm:main4 0} transfers from $\Phi$ to $\Psi$}

With the notation as in the previous sections (including Definition~\ref{def:orbit}), we prove the following result.

\begin{lemma}
\label{lem:C2}
If there exists a $K$-point with a Zariski dense orbit in $G'=G_1'\times G'_3\times G'_4$ under the action of $(\Phi_1, \varphi_2, \varphi_3)$, then there exists a $K$-point with a Zariski dense orbit in $G$ under the action of $\Psi$.
\end{lemma}

\begin{proof}
So, we assume there exists a $K$-point $x\in G'$ with a Zariski dense orbit $\{x_n\}_{n\ge 0}\subset G'(K)$. We let $y\in G(K)$ such that  $g(y)=x$ and we claim that $\OO_\Psi(y)$ is Zariski dense in $G$. Indeed, for each $n\in\N$, using \eqref{eq:almost_commuting_2}, we have that
\begin{equation}
\label{eq:almost_commuting_3}
g\left(\Psi^n(y)\right)-x_n\in G'[\ell_2].
\end{equation}
So, letting $\tilde{g}:=[\ell_2]_{G'}\circ g$ be the composition of $g$ with the multiplication-by-$\ell_2$ map on $G'$, we obtain a finite regular map $\tilde{g}:G\lra G'$. Equation~\eqref{eq:almost_commuting_3} yields that 
\begin{equation}
\label{eq:almost_commuting_4}
\tilde{g}\left(\Psi^n(y)\right)=[\ell_2](x_n)\text{ for each }n\ge 1
\end{equation}
and since $\{x_n\}\subset G'$ is Zariski dense, then also the sequence $\{[\ell_2](x_n)\}\subset G'$ is Zariski dense. But then equation~\eqref{eq:almost_commuting_4} yields that the orbit $\OO_\Psi(y)$ must be Zariski dense in $G$ since $\tilde{g}$ is a finite map.

This concludes our proof of Lemma~\ref{lem:C2}.
\end{proof}

%%%%%%%%%%%%%%%%%%%%%%%%%%%%%%%%%%%%%%%%%%%%%%%%%%%%%%%%%%%%%%%%%%%%%%%%%%%%%%%

\subsection{Condition (B) from Theorem~\ref{thm:main4 0} transfers from $\Phi$ to $\Psi$}

\begin{lemma}
\label{lem:C1}
Assume there exists a non-constant rational function $f:G_1'\lra \bP^1$ such that 
\begin{equation}
\label{eq:f-inv}
f \circ \Phi_1 = f.    
\end{equation}
Then there exists a non-constant rational function $f_1:G\lra \bP^1$ such that $f_1\circ \Psi=f_1$.
\end{lemma}

\begin{proof}
Let $\Pi: G'_1 \times G'_3 \times G'_4 \lra G'_1$ be the projection map onto $G'_1$. By the diagram \eqref{diagram:altogether} (see also \eqref{diagram:altogether_2}) we must have 
\begin{equation}
\label{eq: pi-o-h-o-psi}    
\Pi \circ h \circ \Psi = \Phi_1 \circ \Pi \circ h.
\end{equation}
Then letting $f_1:=f\circ \Pi \circ h$ (which is still a non-constant rational function since $\Pi \circ h$ is a dominant morphism), we have that 
\begin{equation}
\label{eq:f1-inv}
f_1\circ \Psi =f_1,
\end{equation}
as desired.
\end{proof}

\begin{remark}
\label{rem:A_transfer}
It is important to note that we will prove that if condition~(B) holds for the dynamical system $(G',\Phi)$, then it actually holds for $(G'_1,\Phi_1)$ (as stated in Lemma~\ref{lem:C1}), which allows us to transfer the same conclusion to the dynamical system $(G,\Psi)$. 
\end{remark}

%%%%%%%%%%%%%%%%%%%%%%%%%%%%%%%%%%%%%%%%%%%%%%%%%%%%%%%%%%%%%%%%%%%%%%%%%%%%%%%

%%%%%%%%%%%%%%%%%%%%%%%%%%%%%%%%%%%%%%%%%%%%%%%%%%%%%%%%%%%%%%%%%%%%%%%%%%%%%%%

\subsection{Condition (C) from Theorem~\ref{thm:main4 0} transfers from $\Phi$ to $\Psi$}

We show that the aforementioned transfer of condition~(C) from Theorem~\ref{thm:main4 0} from the dynamical system $(G',\Phi)$ to the dynamical system $(G,\Psi)$ holds assuming we establish a slightly more precise version of condition~(C) in the case $G'=G'_1\times G'_3\times G'_4$ and $\Phi=(\Phi_1,\varphi_2,\varphi_3)$; so, with the above notation, we prove the following result.

\begin{lemma}
\label{lem:C3}
Let $G$, $G'$, $\ell_2$, $\Phi$, $\Psi$, $g$ be as in diagram~\eqref{diagram:altogether_2} and equation~\eqref{eq:almost_2}. Assume there exist $n_0\in\N$, there exists a semiabelian variety $Z$ of dimension larger than $\trdeg_{\Fpbar}K$ defined over a finite subfield $\Fq$ of $K$ equipped with the Frobenius endomorphism $F:Z\lra Z$ corresponding to $\Fq$, and there exists a group homomorphism $\tau:G'\lra Z$ such that the following diagram 
\begin{equation}
\label{diagram:Z}
\begin{tikzcd}
G' \arrow[r, "\Phi^{n_0}"] \arrow[d, "\tau"] & G' \arrow[d, "\tau"] \\
Z \arrow[r, "F"]& Z
\end{tikzcd}
\end{equation}
is almost commuting, i.e., for each $x\in G'(K)$, we have that (see also Remark~\ref{rem:correspondence}) 
\begin{equation}
\label{eq:Z2}
\left(\tau\circ \Phi^{n_0}\right)(x) - (F\circ \tau)(x)\in Z[\ell_2].
\end{equation}
Then condition~(C) of Theorem~\ref{thm:main4 0} holds for $(G,\Psi)$.
\end{lemma}

\begin{remark}
\label{rem:actually_C}
In our proof of Theorem~\ref{thm:main4 0} for the dynamical system $$\left(G'_1\times G'_3\times G'_4, \left( \Phi_1, \varphi_2, \varphi_3\right)\right),$$
we will show that when condition~(C) holds for this dynamical system, then actually there exists a semiabelian variety $Z$ defined over a finite field,  along with a dominant group homomorphism $\tau_1:G'_3\lra Z$ such that we actually have: 
$$\left(\tau_1 \circ \varphi_2^{n_0}\right)(x) - (F\circ \tau_1)(x)\in Z[\ell_2],$$
for each $x\in G'_3$. Then letting $\tau:=\tau_1\circ \Pi$, where $\Pi$ is the projection of $G'_1\times G'_3\times G'_4$ on the second factor yields the diagram~\eqref{diagram:Z} and equation~\eqref{eq:Z2} from Lemma~\ref{lem:C3}. 

However, for our proof of Lemma~\ref{lem:C3} we do not require the extra information given above that the homomorphism $\tau$ from Lemma~\ref{lem:C3} factors through the projection map $\Pi$. 
\end{remark}

\begin{proof}[Proof of Lemma~\ref{lem:C3}.]
Equation~\eqref{eq:almost_2} and diagram~\eqref{diagram:altogether_2} yield that for any $x\in G(K)$, we have
\begin{equation}
\label{eq:n_0 1}
\left(g\circ \Psi^{n_0}\right)(x) - \left(\Phi^{n_0}\circ g\right)(x)\in G'[\ell_2].
\end{equation}
Composing with $\tau$ on the left of the equation~\eqref{eq:n_0 1} and noting that $\tau:G'\lra Z$ is a group homomorphism, we get that for each $x\in G(K)$, we have
\begin{equation}
\label{eq:n_0 2}
\left(\tau\circ g\circ \Psi^{n_0}\right)(x) - \left(\tau\circ \Phi^{n_0}\circ g\right)(x)\in Z[\ell_2].
\end{equation}
On the other hand, equation~\eqref{eq:Z2} applied to the point $g(x)\in G'(K)$ yields that
\begin{equation}
\label{eq:n_0 3}
\left(\tau\circ \Phi^{n_0}\right)(g(x))-\left(F\circ \tau\right)(g(x))\in Z[\ell_2].
\end{equation}
So, combining equations~\eqref{eq:n_0 2} and \eqref{eq:n_0 3} yields
\begin{equation}
\label{eq:n_0 4}
\left(\tau\circ g\circ \Psi^{n_0}\right)(x) - \left(F\circ \tau\circ g\right)(x)\in Z[\ell_2].
\end{equation}
We let $\tilde{\tau}:=[\ell_2]_Z\circ \tau\circ g$, which is a dominant group homomorphism $G\lra Z$. Equation~\eqref{eq:n_0 4} yields that
$$\left(\tilde{\tau}\circ \Psi^{n_0}\right)(x)=\left(F\circ \tilde{\tau}\right)(x),$$
for each $x\in G(K)$, and thus, the following diagram is commutative:
\begin{equation}
\begin{tikzcd}
G \arrow[r, "\Psi^{n_0}"] \arrow[d, "\tilde{\tau}"] & G \arrow[d, "\tilde{\tau}"] \\
Z \arrow[r, "F"]& Z
\end{tikzcd}
\end{equation}
as desired in the conclusion of Lemma~\ref{lem:C3}.
\end{proof}

%%%%%%%%%%%%%%%%%%%%%%%%%%%%%%%%%%%%%%%%%%%%%%%%%%%%%%%%%%%%%%%%%%%%%%%%%%%%%%%

\subsection{Theorem~\ref{thm:main4 0} for the simplified dynamical system $(G',\Phi)$}

Using lemmas \ref{lem:C1}, \ref{lem:C2} and \ref{lem:C3} we obtain that Theorem \ref{thm:main4 0} follows from proving its conclusion for the dynamical system $$\left(G'_1\times G'_3\times G'_4,\left(\Phi_1, \varphi_2, \varphi_3\right)\right),$$
as described by the following Theorem.

\begin{theorem}
\label{thm:main4 0 2}
Let $K$ be an algebraically closed field of positive transcendence degree over $\Fpbar$, let $G=G_0\times G_1\times G_2$ be a product of reduced split semiabelian varieties defined over $\Fpbar$, where for each $j=0,1,2$, we have:
$$G_j:=\prod_{i=1}^{r} C_i^{k_{j,i}}$$
for some integers $r$ and $k_{j,i}$, along with some simple semiabelian varieties $C_i$ such that $C_i$ and $C_{i'}$ are non-isogenous for $i\ne i'$. Note that we are allowing the integers $k_{i,j}$ to possibly be equal to zero in which case $C_i^{k_{i, j}}$ is equal to the trivial group. We let $\beta\in G_0(K)$ and also let $\varphi_0\in\End(G_0)$ be a group endomorphism corresponding to a direct sum of unipotent matrices
\begin{equation}
\label{eq:unipotent matrices}
B_{0,1}\oplus B_{0,2}\oplus \cdots \oplus B_{0,r},
\end{equation}
where each $B_{0,i}\in M_{k_{0,i},k_{0,i}}(\End(C_{i}))$ is a direct sum of unipotent Jordan canonical matrices (note that $B_{0,i}$ could possibly be a $0$-by-$0$ matrix, i.e., it may be absent from the direct sum \eqref{eq:unipotent matrices} if $k_{0,i}=0$) of the form 
\[
J_{1, i_{0,1}^{(j)}} \bigoplus J_{1, i_{0,2}^{(j)} - i_{0,1}^{(j)}} \bigoplus \cdots \bigoplus J_{1, i_{0, \ell_j}^{(j)} - i_{0, \ell_j-1}^{(j)}}
.
\] 
where $\ell_j \in \N$ and $i_{0,1}, \dots, i_{0, \ell_j}$ are positive integers such that 
\[
0 < i_{0,1} < i_{0, 2} < \cdots < i_{0, \ell_j} = k_{0, j}.
\]
We let $\Phi_1:=\tau_{\beta}\circ \varphi_0$, i.e., the composition of $\varphi_0$ with the translation-by-$\beta$ map on $G_0$.

We let $\varphi_1\in\End(G_1)$ be a group endomorphism  corresponding to a direct sum of matrices
$$B_{1,1}\oplus B_{1,2}\oplus \cdots \oplus B_{1,r},$$
where each $B_{1,j}\in M_{k_{1,j},k_{1,j}}(\End(C_{j}))$ is a direct sum of Jordan canonical matrices of the form:
$$
J_{F_{C_{j}}^{n_{1}^{(j)}},i_{1,1}^{(j)}}\oplus J_{F_{C_{j}}^{n_{2}^{(j)}},i_{1,2}^{(j)}-i_{1,1}^{(j)}}\oplus \cdots \oplus J_{F_{C_{j}}^{n_{s_j}^{(j)}},i_{1, s_j}^{(j)}-i_{1, s_j-1}^{(j)}}
$$
in which $F_{C_{j}}$ is the image in $\End(C_{j})$ of the Frobenius corresponding to the semiabelian $C_{j}$, while $s_j,  n_1^{(j)},\dots,n_{s_j}^{(j)}\in\N$ and also, $i_{1, 1}^{(j)}, \dots, i_{1, s_j}^{(j)}$ are positive integers such that  
$$0 < i_{1,1}^{(j)} < i_{1,2}^{(j)} < \cdots < i_{1, s_j}^{(j)} = k_{1,j}.$$
for each $j=1,\dots, r$. (Again, it is possible for $B_{1,j}$ to be a $0$-by-$0$ matrix when $k_{1,j}=0$.)

We let $\varphi_2:G_2\lra G_2$ be a finite-to-finite map corresponding to a direct sum of matrices
$$B_{2,1}\oplus B_{2,2}\oplus \cdots \oplus B_{2,r_2},$$
where each  $B_{2,i}\in M_{k_{2,i},k_{2,i}}(\frac{1}{m}\cdot \End(C_{i}))$ for some given $m\in\N$. Furthermore, we assume that each matrix $B_{2,i}$ is either a $0$-by-$0$ matrix whenever $k_{2,i}=0$, or it is NFP (according to Definition~\ref{def:NFP}). 

We let $\Phi:=\Phi_1\times \varphi_1\times \varphi_2$ be the given correspondence on $G=G_0\times G_1\times G_2$. Then at least one of the following statements must hold: 
\begin{itemize}
\item[(A)] there exists a point $\alpha\in G(K)$ such that an orbit of $\alpha$ under $\Phi$ (see also the convention from Definition~\ref{def:orbit} regarding the orbit for a finite-to-finite map) is Zariski dense in $G$.
\item[(B)] there exists a non-constant rational function $f:G_0\lra \bP^1$ such that $f\circ \Phi_1=f$.
\item[(C)]  there exists a semiabelian variety $Z$ defined over a finite subfield $\Fq$ of $K$, endowed with the Frobenius endomorphism $F:Z\lra Z$ corresponding to $\Fq$, such that
\begin{itemize}
\item $\dim(Z)>\trdeg_{\Fpbar}K$; and 
\item there exists a dominant group homomorphism $\tau:G\lra Z$ and there exist positive integers $\ell_0$ and $n_0$ such that for each $x\in G(K)$, we have
$$\left(\tau\circ \Phi^{n_0}\right)(x)-\left(F\circ \tau\right)(x)\in Z[\ell_0],$$
i.e., the following diagram is almost commutative:
\begin{equation}
\begin{tikzcd}
G \arrow[r, "\Phi^{n_0}"] \arrow[d, "\tau"] & G \arrow[d, "\tau"] \\
Z \arrow[r, "F"]& Z
\end{tikzcd}
\end{equation}
\end{itemize}
\end{itemize}  
\end{theorem}

The remaining Sections are devoted to proving Theorem~\ref{thm:main4 0 2}, which in turn yields our main result (Theorem~\ref{thm:main4 0}).

%%%%%%%%%%%%%%%%%%%%%%%%%%%%%%%%%%%%%%%%%%%%%%%%%%%%%%%%%%%%%%%%%%%%%%%%%%%%%%%%
%%%%%%%%%%%%%%%%%%%%%%%%%%%%%%%%%%%%%%%%%%%%%%%%%%%%%%%%%%%%%%%%%%%%%%%%%%%%%%%%

\section{The unipotent case}
\label{sec:unipotent}

In this Section we prove a special case of Theorem~\ref{thm:main4 0 2}, i.e., with the notation as in Theorem~\ref{thm:main4 0 2}, the semiabelian varieties $G_1$ and $G_2$ are trivial. So, we are dealing now with the unipotent case (see \eqref{eq:unipotent matrices}). Also, to simplify our notation later, we introduce the following convention: for a simple semiabelian variety $C$ and some $k\in\N$, each group endomorphism $\varphi\in \End(C^k)$ is identified with a $k$-by-$k$ matrix $Q$ whose entries are in $\End(C)$ and so, for a point $\gamma\in C^k$, we denote
\begin{equation}
\label{eq:notation vector morphism}
\varphi(\gamma):=\gamma^Q.
\end{equation}
Also, in order to emphasize the fact that $\gamma\in C^k$ corresponds to a $k$-tuple $(\gamma_1,\dots, \gamma_k)\in C^k$, we often employ the notation $\vec{\gamma}$ to denote the point $\gamma\in C^k$. So, in particular, the translation-by-$\vec{\gamma}$ map on $C^k$ is denoted by $\tau_{\vec{\gamma}}$. Finally, for a $k$-tuple of endomorphisms 
\begin{equation}
\label{eq:notation vector morphism 2}
\vec{\varphi}:=(\varphi_1,\dots,\varphi_k)\in (\End(C))^k,
\end{equation} 
we let
\begin{equation}
\label{eq:notation vector morphism 3}
\vec{\gamma}^{\vec{\varphi}} :=\sum_{i=1}^k \varphi_i(\gamma_i);
\end{equation}
we will use the notation~\eqref{eq:notation vector morphism 2}~and~\eqref{eq:notation vector morphism 3} for an arbitrary semiabelian variety $C$ (not necessarily simple).

Before proving Proposition~\ref{prop:unipotentMain4}, we first recall the definition of upper asymptotic density of a subset of non-negative integers. 
\begin{definition}
Given a subset $U$ of the set of non-negative integers, the upper asymptotic density of $U$ is given by 
$$\limsup_{m\to\infty} \frac{\#\left\{0\le n\le m\colon n\in U\right\}}{m}.$$
\end{definition}
\begin{remark}
\label{rem:density}
Upper asymptotic densities will appear frequently in the rest of the paper. So, from now on, for the sake of simplifying our notation, we will refer to the upper asymptotic density of some subset $U\subseteq \N_0$ simply  as \emph{density} of $U$ and also, denote it by $d(U)$.
\end{remark}

\begin{proposition}
\label{prop:unipotentMain4}
Let $G=\prod_{i=1}^r C_i^{k_i}$ be a reduced split semiabelian variety (i.e., the $C_i$'s are simple non-isogenous semiabelian varieties defined over $\Fpbar$). Let $K$ be an algebraically closed field, which is transcendental over $\Fpbar$ and  let $\vec{\beta}_i\in C_i^{k_i}(K)$ for $i=1,\dots,r$.  Let $\Phi: G \lra G$ be given by
\begin{equation}
\label{eq:form unipotent phi}
\left(\vec{x}_1, \dots, \vec{x}_r\right) \longmapsto \left(\vec{\beta}_1 +  \vec{x}_1^{Q_1}, \dots, \vec{\beta}_r + \vec{x}_r^{Q_r} \right),
\end{equation}
where $Q_i$ are $k_i$-by-$k_i$ matrices with entries in $M_{k_i,k_i}(\End(C_i))$. Moreover, assume that for $0 \le j \le r$, $Q_j:=J_{1, i_1^{(j)}} \bigoplus J_{1, i_2^{(j)} - i_1^{(j)}} \bigoplus \cdots \bigoplus J_{1, i_{\ell_j}^{(j)} - i_{\ell_j-1}^{(j)}}$ (where $1\le i_1^{(j)}<i_2^{(j)}<\cdots <i_{\ell_j}^{(j)}=k_j$) and $\vec{\beta}_j := (\beta_1^{(j)},\dots, \beta_{k_j}^{(j)})\in C_j^{k_j}(K)$. Then, the following statements are equivalent:
\begin{itemize}
\item[(i)] There is a non-constant rational function $f:G\lra \bP^1$ such that $f\circ \Phi=f$.
\item[(ii)] There is no $\alpha \in G(K)$ whose orbit is Zariski dense in $G(K)$.
\item[(iii)] There exists $1 \le j \le r$ such that $\beta_{i_1}^{(j)}, \dots,\beta_{i_{\ell_j}}^{(j)}$ are linearly dependent over $\End(C_j)$.
\end{itemize}
\end{proposition}
\begin{proof}
As noted already in \cite{A-C, M-S, BGZ}, we have that (i)$\Rightarrow$(ii). 

Now, in order to prove that (ii)$\Rightarrow$(iii), it suffices to show that if for each $j=1,\dots, r$, we have that 
$$\beta_{i_1^{(j)}}^{(j)}, \beta_{i_2^{(j)}}^{(j)}, \dots, \beta_{i_{\ell_j}^{(j)}}^{(j)}\text{ are are linearly independent over $\End(C_j)$,}$$  
then we can find a point in $G(K)$ with a Zariski dense orbit. 

We let $\Fq$ be a finite subfield of $K$ with the property that each $C_j$ is defined over $\Fq$. For each $j=1,\dots, r$, we denote by $F_{C_j}\in\End(C_j)$ the Frobenius endomorphism corresponding to the field $\Fq$. Also, we denote by $F_G$ the corresponding Frobenius endomorphism for the semiabelian variety $G$; when there is no possibility of confusion, we drop the index and simply denote the Frobenius endomorphism by $F$. Furthermore, we let $\Z[F]$ be the ring of operators (consisting of polynomials in the Frobenius endomorphism with integer coefficients) acting on any semiabelian variety defined over $\Fq$ (in our proof, $\Z[F]$ will act on $G$ and also on each $C_i$ and $C_i^{k_i}$).  

After conjugating $\Phi$ with a suitable translation (which does not change the conclusion of our result, according to \cite[Lemma~3.1]{BGSZ}), we may assume without loss of generality that for every $1 \le j \le r$
\begin{equation}
\label{eq:form beta}
\left(\beta_1^{(j)}, \dots, \beta_{k_j}^{(j)}\right) = \left(1,\dots,1,\beta_{i_1^{(j)}}^{(j)},1,\dots, 1,\beta_{i_{\ell_j}^{(j)}}^{(j)}\right), 
\end{equation}
i.e., $\beta_k=1$ unless $k=i_j$ for some $j=1,\dots, \ell_j$ (this is similar to what we used also in the proof of \cite[Proposition~3.10]{G-Sina-20}). 
For every $1\le j \le r$ we choose a point   
\begin{equation}
\label{eq:form alpha}
\vec{\alpha}_j := \left(\alpha_1^{(j)},\dots,\alpha_{i_1^{(j)} - 1}^{(j)},1,\alpha_{i_1^{(j)}+1}^{(j)}, \dots,\alpha_{i_2^{(j)}-1}^{(j)},1,\dots,\alpha_{i_{\ell_j}^{(j)}-1}^{(j)},1\right),
\end{equation} 
such that $\alpha_{i_1^{(j)}}^{(j)},\dots,\alpha_{i_1^{(j)} -1}^{(j)},\beta_{i_1^{(j)}}^{(j)},\alpha_{i_1^{(j)}+1}^{(j)}, \dots,\alpha_{i_{\ell_j}^{(j)}-1}^{(j)},\beta_{i_{\ell_j}^{(j)}}^{(j)}$ are linearly independent over $\End(C_j)$. Now, for every $1 \le i \le r$ let $\varphi_i: C_i^{k_i} \lra C_i^{k_i}$ be the endomorphism corresponding to the matrix $Q_i$ and $\Phi_i: C_i^{k_i} \lra C_i^{k_i}$ be the endomorphism given by $\tau_{\vec{\beta}_i} \circ \varphi_i$. 

Note that any point $\vec{x} \in G$ can be written as $(\vec{x}_1, \dots, \vec{x}_r)$ where $\vec{x}_i \in C_i^{k_i}$. Define $\varphi: G \lra G$ as $\varphi := (\varphi_1, \dots, \varphi_r)$. Next, we let 
\begin{equation}
\label{eq:minimal_polynomial}
P_{\varphi}(x) = x^m + a_{m-1}x^{m-1} + \cdots + a_1x + a_0
\end{equation} 
be the minimal polynomial of $\varphi$ over $\Z$; since $\varphi$ is a unipotent group endomorphism, then we actually know that $P_{\varphi}(x)=(x-1)^m$. (The only relevant information for our proof regarding $P_{\varphi}(x)$ is its degree $m$.)  We also let 
$$\vec{\alpha} := \left(\vec{\alpha}_1, \dots , \vec{\alpha}_r\right) \in G(K)$$ and also $\vec{\beta}:=\left(\vec{\beta}_1,\dots, \vec{\beta}_r\right)\in G(K)$, while we let 
$$\vec{\rho} = \left(\vec{\alpha}, \varphi(\vec{\alpha}), \dots , \varphi^{m-1}(\vec{\alpha}), \vec{\beta}, \varphi(\vec{\beta}), \dots , \varphi^{m-1}(\vec{\beta})\right),$$ 
and for every $1 \le i \le r$, we let 
\begin{equation}
\label{eq:rho_i}
\vec{\rho}_i = \left(\vec{\alpha}_i, \varphi_i(\vec{\alpha}_i), \dots , \varphi_i^{m-1}(\vec{\alpha}_i), \vec{\beta}_i, \varphi_i(\vec{\beta}_i), \dots , \varphi_i^{m-1}(\vec{\beta}_i)\right).
\end{equation} 
The orbit of $\vec{\alpha}$ under $\Phi$ consists of points of the following form:
\[
\matO_{\Phi}(\vec{\alpha}) = \left\{\varphi^n(\vec{\alpha})+\sum_{i=1}^{n-1} \varphi^i(\vec{\beta}) : n \in \N_0 \right\}.
\]
We claim that the orbit of $\vec{\alpha}$ under $\Phi$ is Zariski dense.  We argue by contradiction, and therefore assume that its Zariski closure $V$ is a proper subvariety of $G$. 

We let $\Gamma\subset G$ be the finitely generated $Z[F]$-module consisting of all elements of the form $\vec{v} \cdot \vec{\rho}$, where $\vec{v}\in\Z[F]^{2m}$ (where $\Z[F]$ is the $\Z$-module spanned by the Frobenius operator which acts on any variety $Y$ defined over $\Fq$).  Clearly, we have that $\OO_\Phi(\vec{\alpha})\subseteq \Gamma$. By \cite[Theorem~B]{Moosa-S} (see also Theorem~\ref{Moosa-Scanlon theorem} and  Section~\ref{sec:M-S}), we know that $V \cap \Gamma$ is a union of finitely many sets of the form
\begin{align}
U:=\vec{\gamma} + \Sigma(\vec{\eta}_1,\dots,\vec{\eta}_t; \delta_1,\dots,\delta_t) + H,  \label{eqn:S-sets}  
\end{align}
(for some $t\in\N$), where there exists some positive integer $m_2$ (see Remark~\ref{rem:M-S}) such that 
\begin{equation}
\label{eq:div-hull}
m_2\cdot\gamma,m_2\cdot\eta_1,\dots,m_2\cdot\eta_t\in \Gamma,
\end{equation} 
the $\delta_j$'s are positive integers, $H$ is a subgroup of $\Gamma$ and 
$$\Sigma(\vec{\eta}_1,\dots,\vec{\eta}_t; \delta_1,\dots,\delta_t):=\left\{
\sum_{j=1}^t F^{\delta_jn_j}\cdot \vec{\eta}_j\colon n_j\in\N_0\text{ for }j=1,\dots, t\right\}.$$

Because $\OO_{\Phi}(\vec{\alpha})$ is contained in finitely many sets of the form \eqref{eqn:S-sets}, then there must exist a given set $U$ of the form \eqref{eqn:S-sets} for which the following subset of $\N_0$:  
$$S=\left\{n\in\N_0\colon \Phi^n(\vec{\alpha})\in U\right\}$$
has positive density $d(S)$.

The algebraic closure of $H$ must be an algebraic group $\overline{H}$ contained in the stabilizer of the variety $W$, which is the Zariski closure of $U$. Since $V$ is a proper subvariety and $W\subseteq V$, then $\overline{H}$ must also be a proper algebraic subgroup of $G$. So, there must exist vectors $\vec{\sigma_i} = (\sigma_1^{(i)}, \dots, \sigma_{k_i}^{(i)}) \in \End(C_i)^{k_i}$, not all zero, such that 
\begin{equation}
\label{eq:v kills H}
\left(\vec{\epsilon_i}\right)^{\vec{\sigma_i}} = 1\text{ for each } (\epsilon_1, \dots, \epsilon_r)\in H, \text{ and } 1 \le i \le r\text{ (see \eqref{eq:notation vector morphism 3}).} 
\end{equation}
Let $n\in S$; so,  $\Phi^n(\alpha) \in U$ (see \eqref{eqn:S-sets}). Equation \eqref{eq:div-hull} yields that 
$$m_2\cdot \gamma= \vec{c} \cdot \vec{\rho} \text{ and }m_2 \cdot \eta_i = \vec{b_i} \cdot \vec{\rho} \text{ for each }i=1,\dots, t,$$
where $\vec{c},\vec{b_1},\dots,\vec{b_t}\in \Z[F]^{2m}$ and so,
\begin{equation}
\label{eq:n in}
m_2\cdot \Phi^n(\alpha)=  \Big({\vec{c} +\sum_{j=1}^t F^{\delta_jn_j}\vec{b_j}}\Big) \cdot \vec{\rho} + \vec{u}_n
\end{equation}
for some nonnegative integers $n_j$ and some $\vec{u}_n\in H$. So, combining \eqref{eq:n in} with \eqref{eq:v kills H} and \eqref{eq:rho_i} yields that 
\begin{equation}
\label{eq:n in 2}
(m_2\cdot\Phi_i^n\left(\vec{\alpha}_i\right))^{\vec{\sigma_i}}= \Big( \Big({\vec{c} +\sum_{j=1}^t F^{\delta_jn_j}\vec{b_j}}\Big) \cdot \vec{\rho}_i \Big)^{\vec{\sigma_i}}.
\end{equation}

On the other hand, we know that $\Phi_i^n\left(\vec{\alpha}_i\right)=  \left(\vec{\beta}_i\right)^{\left(\sum_{j=0}^{n-1}Q_i^j\right)} + \left(\vec{\alpha}_i\right)^{\left(Q_i^n\right)}$. Since not all the vectors $\vec{\sigma_1}, \dots, \vec{\sigma_r}$ are equal to zero we have $\vec{\sigma_{s}} \not=0$ for some $1 \le s \le r$. 
We also compute:
\begin{align}
Q_s^n = \begin{pmatrix}
1 &  \binom{n}{1}& \cdots & \binom{n}{i_1^{(s)} - 1} \\
0 & 1 & \cdots & \binom{n}{i_1^{(s)} - 2}\\
\vdots & \vdots & \ddots & \vdots \\
0 & 0 & \cdots & 1
\end{pmatrix} \bigoplus \cdots \bigoplus \begin{pmatrix}
1 &  \binom{n}{1}& \cdots & \binom{n}{i_{\ell_s}^{(s)} - i_{\ell_s-1}^{(s)} - 1} \\
0 & 1 & \cdots & \binom{n}{i_{\ell_s}^{(s)} -i_{\ell_s-1}^{(s)} - 2}\\
\vdots & \vdots & \ddots & \vdots \\
0 & 0 & \cdots & 1
\end{pmatrix} \label{eqn: A^N}
\end{align}
and so,
\begin{align}
Q_s^{n-1} + \cdots + \id = \begin{pmatrix}
n &  \binom{n}{2}& \cdots & \binom{n}{i_1^{(s)}} \\
0 & n & \cdots & \binom{n}{i_1^{(s)}- 1}\\
\vdots & \vdots & \ddots & \vdots \\
0 & 0 & \cdots & n
\end{pmatrix} \bigoplus \cdots \bigoplus \begin{pmatrix}
n &  \binom{n}{2}& \cdots & \binom{n}{i_{\ell_s}^{(s)} - i_{\ell_s-1}^{(s)}} \\
0 & n & \cdots & \binom{n}{i_{\ell_s}^{(s)} - i_{\ell_s-1}^{(s)} - 1}\\
\vdots & \vdots & \ddots & \vdots \\
0 & 0 & \cdots & n
\end{pmatrix}. \label{eqn: sumA^N}
\end{align}
Therefore, using \eqref{eq:n in 2} along with formulas \eqref{eqn: A^N} and \eqref{eqn: sumA^N}, we obtain that for each $n\in S$, we have
\begin{align}
m_2 \cdot \left(\vec{\beta}_s\right)^{\left(\sum_{j=0}^{n-1}Q_s^j\right)^t \cdot \vec{\sigma_s}} + m_2\cdot (\vec{\alpha}_s)^{\left(Q_s^n\right)^t\cdot \vec{\sigma}_s}=\Big( \Big({\vec{c} +\sum_{j=1}^r F^{\delta_jn_j}\vec{b_j}}\Big)\cdot \vec{\rho}_s\Big)^{\vec{\sigma}_s}. \label{eqn:peqn}
\end{align}
Now, both sides in \eqref{eqn:peqn} consist of a $\End(C_s)$-linear combination of 
\begin{equation}
\label{eq:a and b}
\alpha_1^{(s)},\dots, \alpha_{i_1^{(s)}-1}^{(s)},\beta_{i_1^{(s)}}^{(s)},\alpha_{i_1^{(s)}+1}^{(s)},\dots, \alpha_{i_2^{(s)}-1}^{(s)},\beta_{i_2^{(s)}}^{(s)},\alpha_{i_2^{(s)}+1}^{(s)},\dots, \alpha_{i_{\ell_s}^{(s)}-1}^{(s)},\beta_{i_{\ell_s}^{(s)}}^{(s)}
\end{equation}
and since the $k_s$ elements of $C_s(K)$ from \eqref{eq:a and b} are linearly independent over $\End(C_s)$, then it means that the coefficient of each $\alpha_i^{(s)}$ and each $\beta_{i_j^{(s)}}^{(s)}$ appearing in the left-hand side of \eqref{eqn:peqn} must match the corresponding coefficient of the $\alpha_i^{(s)}$, respectively of $\beta_{i_j^{(s)}}^{(s)}$ appearing in the right-hand side of \eqref{eqn:peqn}. 
   
Now, since $\vec{\sigma_s}:=\left(\sigma_1^{(s)},\dots, \sigma_{k_s}^{(s)}\right)$ is nonzero, then there is some $1 \le k \le \ell_s$ such that the tuple  $\left(\sigma_{i_{k-1}^{(s)}+1}^{(s)},\dots,\sigma_{i_{k}^{(s)}}^{(s)}\right)$ is nonzero (where we denoted $i_0:=0$ for convenience). We use equations \eqref{eqn: A^N} and \eqref{eqn: sumA^N} to compute the coefficient of $\beta_{i_k}$ appearing in the left-hand side of \eqref{eqn:peqn} and then comparing it with the coefficient of $\beta_{i_k}$ from the right-hand side of \eqref{eqn:peqn}, we get 
$$
m_2\left(\sigma_{i_{k-1}^{(s)}+1}^{(s)}\cdot \binom{n}{i_k^{(s)} - i_{k - 1}^{(s)}} + \sigma_{i_{k-1}^{(s)}+2}^{(s)}\cdot \binom{n}{i_k^{(s)}-i_{k-1}^{(s)}-1} + \cdots + \sigma_{i_k^{(s)}}^{(s)}\cdot \binom{n}{1}\right)$$
\begin{equation} 
= \tau_0 + \sum_{j=1}^{r}\tau_j \cdot F_{C_s}^{\delta_jn_j},
\label{eq:den 0 equ}
\end{equation}
for some endomorphisms $\tau_0,\dots, \tau_r$ in $\End(C_s)$ (which are independent of $n$ and instead, they only depend on the coordinates of the vectors $\vec{c},\vec{b_1},\dots, \vec{b_r}$ and the entries of the vector $\vec{\sigma}_s$). Now, note that $\End(C_s) \tensor \Q(F_{C_s})$ is a finite-dimensional $\Q(F_{C_s})$-vector space with basis $\mathcal{B} = \{\psi_1, \dots, \psi_{h}\}$. Since the tuple $\left(\sigma_{i_{k-1}^{(1)}+1},\dots, \sigma_{i_k^{(1)}}\right)$ is nonzero, then there must exist $1 \le h_0 \le h$ such that the $\psi_{h_0}$-th coordinates of $\sigma_{i_{k-1}^{(1)}+1},\dots, \sigma_{i_k^{(1)}}$ with respect to the basis $\mathcal{B}$ are not all equal to zero. The coefficient of $\psi_{h_0}$ in the left hand side of equation \eqref{eq:den 0 equ} is equal to
\begin{equation}
\label{eq:density P equ}
P(n) := m_2\cdot \sum_{j=1}^{i_k^{(s)}-i_{k-1}^{(s)}} u_{i_{k-1}^{(s)}+j}\cdot \binom{n}{i_k^{(s)}-i_{k-1}^{(s)}+1-j}
\end{equation}
where, $P \in \Q[x]$ is non-constant and $u_{i_{k-1}^{(s)}+j}$ is the $\psi_{h_0}$-th coordinate of $\sigma_{i_{k-1}^{(s)}+j}$. So, equations \eqref{eq:density P equ} and \eqref{eq:den 0 equ} yield that each element $n\in S$ must satisfy an equation of the form:
\begin{equation}
\label{eq:density form equ}
P(n)=c_0+\sum_{j=1}^r c_jF_{C_s}^{\delta_jn_j},
\end{equation}
for some $n_j\in\N_0$, where the $c_i$'s are the $\psi_{h_0}$-th coordinates of $\tau_0,\tau_1,\dots, \tau_r$. Because $P \in \Z[x]$ is non-constant (while the $\delta_j$'s are positive integers and the $c_j$'s are given), \cite[Theorem~1.1]{GOSS2} yields that $d(S)=0$, therefore contradicting our assumption that $S$ has positive density. Hence, indeed $\OO_\Phi(\alpha)$ must be Zariski dense in $G$ which shows the implication  (ii)$\Rightarrow$(iii). 

Finally, in order to prove that (iii)$\Rightarrow$(i), we know that there exist endomorphisms $\sigma_1, \dots, \sigma_{\ell_j} \in \End(C_j)$ which are not all equal to zero and  $\sum_{k = 1}^{\ell_j} \sigma_k\Big(\beta_{i_k^{(j)}}^{(j)}\Big) = 0$. Let $N:= k_1 + \cdots + k_r$ and $L = k_1 + \cdots + k_{j - 1}$ and consider the morphism $f: G \lra C_j$ given by
$$
(x_1, \dots, x_{N}) \mapsto \sum_{k = 1}^{\ell_j} \sigma_k\Big(x_{i_k^{(j)} + L}\Big),
$$
where we represented each element $x\in G=\prod_{i=1}^r C_i^{k_i}$ as $(x_1,\dots, x_N)$. 
Then $f$ is clearly a non-constant group homomorphism (and thus, a dominant homomorphism since $C_j$ is a simple semiabelian variety), as not all of the $\sigma_i$'s are equal to zero; furthermore, $f \circ \Phi = f$. So, if $\chi: C_j \lra \bP^1$ is any non-constant rational function, then $\chi \circ f: G \lra \bP^1$ would be a non-constant rational which is invariant under $\Phi$. This concludes our proof for Proposition~\ref{prop:unipotentMain4}.
\end{proof}

%%%%%%%%%%%%%%%%%%%%%%%%%%%%%%%%%%%%%%%%%%%%%%%%%%%%%%%%%%%%%%%%%%%%%%%%%%%%%%%
%%%%%%%%%%%%%%%%%%%%%%%%%%%%%%%%%%%%%%%%%%%%%%%%%%%%%%%%%%%%%%%%%%%%%%%%%%%%%%%

\section{A mixed case}
\label{sec:non-unipotent}

In this Section, we extend Proposition~\ref{prop:unipotentMain4} by allowing also a non-unipotent part in the map $\Phi:G\lra G$ (where $G$ is a semiabelian variety defined over $\Fpbar$). Actually, in this special case, we consider even the case when $\Phi$ is only a finite-to-finite map; this result will be instrumental in deriving the general conclusion from Theorem~\ref{thm:main4 0 2}.

\begin{proposition}
\label{prop:split}
Let $G := G_1 \times G_2$ where 
\[
G_1 := \prod_{i=1}^r C_i^{k_i}, \quad G_2 := \prod_{i=1}^r C_i^{k'_i},
\]
and $C_1, \dots, C_r$ are non-isogenous simple semiabelian varieties defined over $\Fpbar$.  (Note that we are allowing $k_1, \dots, k_r, k'_1, \dots, k'_r$ to be equal to zero keeping in mind that $C_i^{0}$ represents the trivial group.) Let $K$ be an algebraically closed field of positive transcendence degree over $\Fpbar$. For every $j=1,\dots,r$, let $Q_j$ be a unipotent $k_j$-by-$k_j$ matrix in Jordan canonical form, i.e.,
$$Q_j:=J_{1,i_1^{(j)}}\oplus J_{1,i_2^{(j)}-i_1^{(j)}}\oplus \cdots \oplus J_{1,i_{s_j}^{(j)}-i_{s_j-1}^{(j)}},$$
where $1\le i_1^{(j)}<i_2^{(j)}<\cdots <i_{s_j}^{(j)}=k_j$, and let $\vec{\beta_j}:=\left(1,\dots, 1,\beta_{i_1^{(j)}}^{(j)},1,\dots, 1,\beta_{i_{s_j}^{(j)}}^{(j)}\right)\in C_j^{k_j}(K)$. We let   
$$\vec{\gamma_j}:=\left(\gamma_1^{(j)},\dots, \gamma_{i_1^{(j)}-1}^{(j)},1,\gamma_{i_1^{(j)}+1}^{(j)},\dots, \gamma_{i_2^{(j)}-1}^{(j)},1,\gamma_{i_2^{(j)}+1}^{(j)},\dots, \gamma_{i_{s_j}^{(j)}-1}^{(j)},1\right)\in C_j^{k_j}(K)$$
and let $\vec{\alpha_j}:=\left(\alpha_1^{(j)},\dots, \alpha_{k'_j}^{(j)}\right)\in C_j^{k'_j}(K)$. Assume  the following elements of $C_j(K)$ are linearly independent over $\End(C_j)$:
\begin{equation}
\label{eq:large group}
\gamma_1^{(j)},\dots, \gamma_{i_1^{(j)}-1}^{(j)},\beta_{i_1^{(j)}}^{(j)},\gamma_{i_1^{(j)}+1}^{(j)},\dots, \gamma_{i_{s_j}^{(j)}-1}^{(j)},\beta_{i_{s_j}^{(j)}}^{(j)},
\end{equation} 
Also, assume that the $\alpha_i^{(j)}$'s (the coordinates of $\vec{\alpha}_j$) are linearly independent from the elements from \eqref{eq:large group} over $\End(C_j)$, i.e., letting $\Gamma_j$ and $\Lambda_j$ be the subgroups of $C_j(K)$ spanned by the action of the elements of $\End(C_j)$ on the elements from \eqref{eq:large group} and on the  $\alpha_i^{(j)}$'s, respectively, we must have $\Gamma_j\cap\Lambda_j=\{0\}$ for every $1 \le j \le r$.

Let $\Phi_1:G_1 \lra G_1$ be the regular map defined by
$$\left(\vec{x}_1, \dots, \vec{x}_r\right) \longmapsto \left(\vec{\beta}_1 +  \vec{x}_1^{Q_1}, \dots, \vec{\beta}_r + \vec{x}_r^{Q_r} \right).$$
Let $\Phi_2$ be a finite-to-finite map from $G_2$ to $G_2$ corresponding to matrices $Q'_1, \dots, Q'_r$ where $Q'_i \in M_{k'_i}(\End^0(C_i))$ for $i=1,\dots,r$ and suppose that we have the next almost commutative diagram 
\begin{equation}
\label{diagram:g}
\begin{tikzcd}
G' \arrow[r, "\Psi'"] \arrow[d, "g'"'] & G' \arrow[d, "g'"] \\
G_2 \arrow[r, "\Phi_2"] & G_2,
\end{tikzcd}
\end{equation}
where $G'$ is a split semiabelian variety, $\Psi'$ is a group endomorphism of $G'$ and $g': G' \lra G_2$ is an isogeny. Let $\vec{\alpha} := (\vec{\alpha}_1, \dots,\vec{\alpha}_r)$ and $\vec{\gamma} := (\vec{\gamma}_1, \dots, \vec{\gamma}_r)$. Assume that for some given orbit $\{y_n\}_{n \ge 0}$ of $\vec{\alpha}$ under $\Phi_2$, then for any positive density subset $S\subseteq\N_0$, the set
$\left\{y_n \colon n\in S\right\}$
is Zariski dense in $G_2$. Then $\left\{\left(\Phi_1^n(\vec{\gamma}), y_n\right): n \in S\right\}$ is Zariski dense in $G$.
\end{proposition}

\begin{proof}
So, we let $S\subseteq \N_0$ be a positive density subset. 

We recall that since diagram~\eqref{diagram:g} is almost commutative (see also Sections~\ref{sec:correspondences}~and~\ref{sec:almost}), then it means that there exists some positive integer $\ell_2$ such that for each $x\in G'$, we have that
\begin{equation}
\label{eq:x G'}
\left(g'\circ \Psi'\right)(x)-\left(\Phi_2\circ g'\right)(x)\in G_2[\ell_2]. 
\end{equation}
At the expense of replacing $\ell_2$ by a multiple of it, we may also assume that given $\hat{g}':G_2\lra G'$, we also have that
\begin{equation}
\label{eq:dual}
g'\circ \hat{g}'=[\ell_2]_{G_2}\text{ and }\hat{g}'\circ g'=[\ell_2]_{G'}.
\end{equation}
Equations \eqref{eq:x G'} and \eqref{eq:dual} yield that
\begin{equation}
\label{eq:G 0}
[\ell_2]\circ \Phi_2=g'\circ \Psi' \circ \hat{g}'. 
\end{equation}
In particular, equation \eqref{eq:G 0} yields that
\begin{equation}
\label{eq:G n}
[\ell_2]\circ \Phi_2^n=g'\circ \left(\Psi'\right)^n\circ \hat{g}'\text{ for each }n\in\N.
\end{equation}

Next, we consider the following almost commutative diagram 
\begin{equation}
\label{eq:almost_commuting_7}
\begin{tikzcd}
G_1 \times G' \arrow[r, "\Psi"] \arrow[d, "g"'] & G_1 \times G' \arrow[d, "g"] \\
G_1 \times G_2 \arrow[r, "{(\Phi_1, \Phi_2)}"] & G_1 \times G_2,
\end{tikzcd}
\end{equation}
where $g := (\id_{G_1}, g')$ and $\Psi := (\Phi_1, \Psi')$. Choose $\vec{\alpha}_0 \in G_2$ such that 
\begin{equation}
\label{eqn:alpha_0}
[\ell_2]_{G_2}(\vec{\alpha}_0) = \vec{\alpha}
\end{equation}
and let $\vec{x}_0 := (\vec{\gamma}, \vec{\alpha}_0)$. Let 
\[
\OO = \left\{\left(g \circ \Psi^n \circ \hat{g}\right)\left(\vec{x}_0\right): n \in S \right\}\subset \left(G_1\times G_2\right)(K).
\]
where $\hat{g} := (\id_{G_1}, \hat{g}')$. Using \eqref{eq:G n}, \eqref{eq:almost_commuting_7} and \eqref{eqn:alpha_0}, it suffices to prove that $\OO$ is Zariski dense in $G_1\times G_2$. So, we assume otherwise and let $V$ be the Zariski closure of $\OO$; then $V$ is a proper subvariety of $G_1\times G_2$.

Since $\Psi'\in\End(G')$ is integral over $\Z$, combined with the fact that the Frobenius endomorphism $F:G'\lra G'$ (corresponding to $\Fq$) is also integral over $\Z$ (inside $\End(G')$), and furthermore $F$ commutes with $g$ and $g'$ (since $F$ is in the center of $\End^0(G')$), we conclude that 
\begin{equation}
    \Delta := \left\{\left( \sum_{i=0}^{n} a_i\left(g \circ \Psi^i \circ \hat{g}\right)\left(\vec{x}_0\right)\right): n\in\N_0\text{ and }a_i \in \Z[F]\right\} \label{eqn:LambdaForm}
\end{equation}
is a finitely generated $Z[F]$-submodule of $G_1\times G_2=G$, which contains $\OO$. So, using Theorem~\ref{Moosa-Scanlon theorem} and arguing identically as in the proof of Proposition~\ref{prop:unipotentMain4}, we have that $V\cap\Delta$ is a finite union of sets of the form
\begin{align}
U:=\vec{\lambda} + \Sigma(\vec{\eta}_1,\dots,\vec{\eta}_s; \delta_1,\dots,\delta_s) + H,  \label{eqn:S-sets 3}  
\end{align}
where there exists some positive integer $m$ such that 
\begin{equation}
\label{eq:div-hull 3}
m\cdot \vec{\lambda},m\cdot \vec{\eta}_1,\dots,m\cdot \vec{\eta}_s\in \Delta,
\end{equation} 
while the $\delta_j$'s are positive integers and $H$ is a subgroup of $\Delta$. Because $\OO$ is entirely contained in the union of finitely many sets as the one from \eqref{eqn:S-sets 3}, at the expense of replacing $S$ with a subset of positive density, there must exist some set $U$ as in \eqref{eqn:S-sets 3} containing $\left(g \circ \Psi^n \circ \hat{g}\right)\left(\vec{x}_0\right)$ for all integers $n$ in $S$. Since we assumed that  $V$ is a proper subvariety of $G$, then the Zariski closure of $H$ must be a proper algebraic subgroup of $G$; so, there must exist an endomorphism $\sigma:G\lra G$ such that  $H\subseteq \ker(\sigma)$. 

We let $L$ be the degree of the minimal (monic) polynomial $P_{\Psi}$ with integer coefficients for which $P_{\Psi}(\Psi)=0$ in $\End(G_1\times G')$. Then  we let 
\[
\vec{\rho} := \left((g \circ \hat{g}), \left(g \circ \Psi \circ \hat{g}\right), \dots, \left(g \circ \Psi^{L-1} \circ \hat{g}\right)\right),
\]
and
\[
\vec{\rho}(\vec{x}) := \left((g \circ \hat{g})(\vec{x}), \left(g \circ \Psi \circ \hat{g}\right)(\vec{x}), \dots, \left(g \circ \Psi^{L-1} \circ \hat{g}\right)(\vec{x})\right),
\]
for every $x \in G$. Note that any element in $\Delta$ must be a linear combination of the coordinates of $\vec{\rho}(\vec{x}_0)$ over $\Z[F]$. So, using equation \eqref{eq:div-hull 3}, there exist vectors $\vec{u}_0, \vec{u}_1, \dots, \vec{u}_s \in \Z[F]^{L}$ such that 
\begin{align}
&m\cdot\vec{\eta}_i = \vec{u_i} \cdot \vec{\rho}(\vec{x}_0) \text{ for each $i=1,\dots, s$ and  } m\cdot \vec{\lambda} = \vec{u_0} \cdot \vec{\rho}\left(\vec{x}_0\right).  \notag    
\end{align}  
So, for each $n \in S$, using that $\left(g \circ \Psi^n \circ \hat{g}\right)\left(\vec{x}_0\right) \in U$ along with equation \eqref{eqn:S-sets 3}, we must have some some non-negative integers $n_i$ (for $i=1,\dots, s$) such that 
\begin{align}
\sigma\left(m\left(g \circ \Psi^n \circ \hat{g}\right)\left(\vec{x}_0\right)\right) &= \sigma\left(\left( \vec{u}_0\cdot \vec{\rho} + \sum_{i=1}^s F^{n_i\delta_i}(\vec{u}_i\cdot\vec{\rho})\right)\left(\vec{x}_0\right)\right).
\label{eq:lastEquality}
\end{align}
Due to the way the coordinates of $\vec{x}_0$ are chosen we know that there does not exist any non-trivial endomorphism of $G_1 \times G_2$ that vanishes at $\vec{x}_0$. Therefore, we must have   
\begin{equation}
\label{eq:equalEndomorphisms1}
\sigma\left(m\left(g \circ \Psi^n \circ \hat{g}\right)\right) = \sigma\left(\left( \vec{u}_0\cdot \vec{\rho} + \sum_{i=1}^s F^{n_i\delta_i}(\vec{u}_i\cdot\vec{\rho})\right)\right)
\end{equation}
for every $n \in S$ (where the integers $n_i$'s depend on $n$). Now, the group endomorphism $\sigma$ corresponds  to some matrix $P$ whose rows are of the form $\vec{v}_1 \oplus \vec{v}_2$ where $\vec{v}_1 \in \prod_{i=1}^r \End(C_i)^{k_i}$ and $\vec{v}_2 \in \prod_{i=1}^r \End(C_i)^{k'_i}$. We know that $m(\left(g \circ \Psi^n \circ \hat{g}\right))$ corresponds to matrices 
$$m Q_1^n, \dots, m Q_r^n, (m\ell_2)(Q'_1)^n, \dots, (m\ell_2)(Q'_r)^n\text{ (see \eqref{eq:G n})}$$ 
for some fixed positive integer $\ell_2$. 

If $\vec{v}_1$ is a nonzero vector, using our hypothesis that the $\beta_{i_k}^{(j)}$'s and the $\gamma_k^{(j)}$'s are linearly independent over $\End(C_j)$, while the $\alpha_k^{(j)}$'s are linearly independent over $\End(C_j)$ with respect to the $\beta_{i_k}^{(j)}$'s and the $\gamma_k^{(j)}$'s, and arguing exactly as in the proof of Proposition~\ref{prop:unipotentMain4} (see equations \eqref{eq:den 0 equ}, \eqref{eq:density P equ} and \eqref{eq:density form equ}) we get that there exists some non-constant polynomial $P_0$, there exists some $s\in \{1,\dots, r\}$,  and there exist $c_0, c_1, \dots, c_r \in \Q\left(F_{C_s}\right)$ such that for each $n\in S$, there are non-negative integers $n_j$ such that
\begin{equation}
\label{eq:finite digits 2}
P_0(n)=c_0+\sum_{j=1}^r c_jF_{C_s}^{\delta_jn_j},
\end{equation}
Since $S$ has positive density, this yields a contradiction to the conclusion of \cite[Theorem~1.1]{GOSS2}. Therefore, for any row of the matrix $P$ of the form $\vec{v}_1 \oplus \vec{v}_2$, we must have $\vec{v}_1 = 0$; this  holds for any endomorphism that kills all of the elements of $H$. So, if we let $\overline{H}$ be the Zariski closure of $H$, then we must have $\overline{H} = G_1 \times H_2$
where $H_2$ is an algebraic subgroup of $G_2$. 
So, letting $W$ be the Zariski closure of $U$ in $G$, then its stabilizer must contain $\overline{H}$ and therefore, it contains $G_1$ (seen as a subgroup of $G_1$ under the natural embedding $\vec{x}\mapsto \vec{x}\oplus \vec{0}_{G_2}$); i.e., for each $\vec{\epsilon}_1\in G_1$ and each $\vec{\mu}\in W$, we have that $\left(\vec{\epsilon} \oplus \vec{0}_{G_2}\right) + \vec{\mu}\in W$. Hence $W= G_1 \times Z$, for some subvariety $Z\subseteq G_2$. However, $Z$ must contain each $\left(g'\circ \left(\Psi'\right)^n\circ \hat{g}'\right)(\vec{\alpha}_0)$ for $n\in S$. Using equations \eqref{eq:G n} and \eqref{eqn:alpha_0} we must have 
\[
\left(g'\circ \left(\Psi'\right)^n\circ \hat{g}'\right)(\vec{\alpha}_0) - y_n \in Z[\ell_2].
\]
By our hypothesis, $\{y_n\}_{n \in S}$ is Zariski dense in $G_2$ and therefore, 
$$\left\{\left(g'\circ \left(\Psi'\right)^n\circ \hat{g}'\right)(\vec{\alpha}_0): n \in S\right\}$$ 
must actually be Zariski dense in $G_2$, which yields that $Z = G_2$. Thus, $W=G$ and indeed $\OO$ must be Zariski dense in $G$. 

Now, using the fact that \eqref{eq:almost_commuting_7} is almost commutative along with \eqref{eq:G n}, we have that
\begin{equation}
\label{eq:almost_commuting_5}
g\left(\Psi^n(\hat{g}(\vec{x}_0))\right)-(\Phi_1^n(\vec{\gamma}), y_n)\in G[\ell_2].
\end{equation}
 So, letting $\tilde{g}:=[\ell_2]_{G}\circ g$ be the composition of $g$ with the multiplication-by-$\ell_2$ map on $G$, we obtain a finite regular map $\tilde{g}:G_1 \times G'\lra G$. Equation~\eqref{eq:almost_commuting_5} yields that 
\begin{equation}
\label{eq:almost_commuting_6}
\tilde{g}\left(\Psi^n(\hat{g}(\vec{x}))\right)=[\ell_2]_{G}(\Phi_1^n(\vec{\gamma}), y_n) \text{ for each }n\ge 1
\end{equation}
and since $\left\{\left(g \circ \Psi^n \circ \hat{g}\right)\left(\vec{x}_0\right)\colon n\in S\right\}$ is Zariski dense, then also the sequence $\{\left\{\left(\tilde{g} \circ \Psi^n \circ \hat{g}\right)\left(\vec{x}_0\right)\colon n\in S\right\}\}\subset G$ is Zariski dense. But then equation~\eqref{eq:almost_commuting_6} yields that  $\{(\Phi_1^n(\vec{\gamma}), y_n): n \in S\}$ must be Zariski dense in $G$ since $[\ell_2]_{G}$ is a finite map. This concludes our proof of Proposition~\ref{prop:split}. 
\end{proof}

\section{The case of group endomorphisms whose eigenvalues are powers of the Frobenius element in the endomorphism ring}
\label{sec: case 2}

The next result provides the conclusion in Theorem~\ref{thm:main4 0 2} in the case we have a group endomorphism of a split reduced semiabelian variety, whose corresponding eigenvalues are powers of the Frobenius element in the endomorphism ring.

\begin{proposition}
\label{prop:noUnity}
Let  $K$ be an algebraically closed field of characteristic $p$ with transcendence degree $d\ge 1$ over $\Fpbar$, and let $G$ be a reduced split semiabelian variety, i.e.,
$$G:=\prod_{i=1}^r C_i^{k_i},$$
where the $C_i$'s are non-isogenous simple semiabelian varieties defined over some finite subfield $\Fq\subset K$, while the $k_i$'s are positive integers. We let $F:G\lra G$ be the Frobenius endomorphism of $G$ associated to the finite field $\Fq$; also, for each $i=1,\dots, r$, we let  $F_{C_i}\in\End(C_i)$  be the corresponding Frobenius for each semiabelian variety $C_i$. Let  $\Phi: G \lra G$ be a dominant group endomorphism corresponding to matrices $ Q_i \in M_{k_i,k_i}(\End(C_i))$ for $1 \le i \le r$. Assume that each matrix $Q_j$ is a Jordan canonical matrix of the form:  
\begin{equation}
\label{eqn:Q_iForm}
J_{F^{n_{1}^{(j)}},i_1^{(j)}}\oplus J_{F^{n_{2}^{(j)}},i_2^{(j)}-i_1^{(j)}}\oplus \cdots \oplus J_{F^{n_{s_j}^{(j)}},i_{s_j}^{(j)}-i_{s_j-1}^{(j)}},
\end{equation}
where $1\le i_1^{(j)}< i_2^{(j)}<\cdots < i_{s_j}^{(j)}=k_j$, and $s_j$, $n_\ell^{(j)}$ are positive integers. (for $1\le j\le r$ and $1\le \ell\le s_j$). 
 
Then one of the following statements must hold:
\begin{itemize}
\item[(A)] There exists $\vec{\alpha} = (\vec{\alpha}_1, \dots, \vec{\alpha}_r)\in G(K)$ where $\vec{\alpha}_i \in C_i^{k_i}$ for each $1 \le i \le r$, whose orbit under $\Phi$ is Zariski dense in $G$. Furthermore, given any finitely generated submodules $\Gamma_1, \dots, \Gamma_r$ where $\Gamma_i \subset C_i(K)$ is an $\End(C_i)$-submodule for every $1 \le i \le r$, one can choose $\vec{\alpha}\in G(K)$ such that
\begin{itemize}
\item[(i)] the subgroup spanned by the action of the elements of $\End(C_i)$ on $\alpha_i^{(1)},\dots,\alpha_i^{(k_i)}$ (the coordinates of $\vec{\alpha}_i$) has trivial intersection with $\Gamma_i$ for every $1 \le i \le r$; and
\item[(ii)] for any subset $S$ of positive integers with positive density we have $\{\Phi^n(\alpha): n \in S\}$ is Zariski dense in $G$.
\end{itemize}
\item[(B)] There exist $1 \le j_1 \le \cdots \le j_{\ell} \le r$, and $u_1, \dots, u_{\ell}$ satisfying $1 \le u_{k} \le s_{j_k}$ for every $1 \le k \le \ell$, such that the pairs $(j_k, u_k)$ are distinct and we have: 
\[
n_{i_{u_1}}^{(j_1)} = \dots = n_{i_{u_{\ell}}}^{(j_\ell)},
\]
along with  $\sum_{t=1}^\ell \dim\left(C_{j_t}\right)  > d$. 
\end{itemize}
\end{proposition}

\begin{proof}
In our proof,  by convention, we let $i_0^{(j)}=0$ for each $1\le j\le r$. 

If conclusion $(B)$ holds then we are done. So, assume from now on, that conclusion~$(B)$ does not hold. In particular, we have some finitely generated subgroups $\Gamma_i$ as in conclusion~(A) from Proposition~\ref{prop:noUnity}.

Each component $C_j$ of $G$ is embedded in $\Pl^{N_j}$ (with coordinate axes labeled $x_i$ for $1\le i\le N_j+1$) for some $N_j \in \N$. We let $d_j := \dim(C_j)$ for every $j = 1,\dots, r$; without loss of generality, we assume each $C_j$ projects dominantly onto the first $d_j$ coordinates of $\Pl^{N_j}$, 
i.e., the projection $\left(x_1:x_2\cdots :x_{N_j+1}\right)\mapsto \left(x_1:x_2:\cdots :x_{d_j}:x_{N_j+1}\right)$ induces a dominant rational map $\pi_j:C_j\dra \mathbb{P}^{d_j}$.

We let 
\[
\mathcal{P} = \left\{ (j, \ell): 1 \le j \le r, 1 \le \ell \le s_j \right\}, 
\]
be a totally ordered set with the usual lexicographical order. We can partition $\mathcal{P}$ by the sets 
\begin{equation}
\label{eq:def_Pm}
\mathcal{P}_n = \left\{ (j, \ell): 1 \le j \le r, 1 \le \ell \le s_j, n_\ell^{(j)} = n \right\}
\end{equation}
and we extend the lexicographic order on each $\mathcal{P}_n$.

Let $\{t_1, \dots, t_{d}\} \subset K$ be an arbitrary algebraically independent set over $\Fpbar$ (i.e., a transcendence basis for $K/\Fpbar$). 

For every $j=1,\dots,r$ and each $\ell=1,\dots,s_j$ there must exist a unique $n \in \N$ such that $(j, \ell) \in \mathcal{P}_n$. Letting $(j_1, \ell_1), \dots, (j_u, \ell_u)$ be all the pairs in $\mathcal{P}_n$ that are smaller than $(j, \ell)$, with respect to the lexicographical order imposed on $\mathcal{P}$, we can define 
\begin{equation}
\label{eq:S-l-j}
S_{\ell, j} := d_{j_1} + \cdots + d_{j_u}, 
\end{equation}
where $S_{\ell, j}$ is defined to be equal to zero whenever $(j,\ell)$ is the smallest pair in $\mathcal{P}_n$.  We also let 
\begin{equation}
\label{eq:t-l-j}
t_{k, \ell, j} := t_{S_{\ell, j} + k}
\end{equation}
for every $1 \le k \le d_j$ (and each $(j,\ell)\in \mathcal{P}$). Note that since condition (B) is not met we must have $S_{\ell, j} + k \le d$ for every possible choice of $\ell, j, k$ which means that $t_{k, \ell, j}$ is well-defined. Indeed, the fact that $(j,\ell)$ along with each $(j_t,\ell_t)$ for $1\le t\le u$ are contained in $\mathcal{P}_n$ (where the pairs $(j_t,\ell_t)$ are all the pairs contained in $\mathcal{P}_n$ smaller than $(j,\ell)$) means that
$$n_\ell^{(j)}=n_{\ell_1}^{(j_1)}=\cdots = n_{\ell_u}^{(j_u)}$$
and so, our assumption that condition (B) from Proposition~\ref{prop:noUnity} does not hold yields that
$$d_j+d_{j_1}+\cdots + d_{j_u}\le d,$$
as desired. 

Next, for each $j=1,\dots, r$, and for each $S:=S_{\ell,j}$ (for some $1\le \ell\le s_j$), we choose a point $\alpha_{j}^{(S)}\in C_j(K)$ (note that $\pi_j$ is a dominant map and $\left(t_{S+1}:t_{S+2}:\dots : t_{S+d_j}:1\right)$ is a generic point for $\bP^{d_j}_{\Fpbar}$) such that:
\begin{equation}
\label{eqn:alphaCoord}
\pi_j\left(\alpha_j^{(S)}\right):=\left(t_{S+1}:t_{S+2}:\dots : t_{S+d_j}:1\right)
\end{equation}
Then for each $j=1,\dots, r$ and for each $\ell=1,\dots, s_j$, we let
\begin{equation}
\label{eq:same points}
\alpha_{\ell,j}:=\alpha_j^{(S_{\ell,j})}.
\end{equation} 
Also, recalling that $t_{k,\ell,j}:=t_{S_{\ell,j}+k}$ for each $j=1,\dots, r$, each $\ell=1,\dots, s_j$ and each $k=1,\dots, d_j$, then we see that
\begin{equation}
\label{eq:alphaCoord 2}
\pi_j\left(\alpha_{\ell,j}\right)=\left(t_{1,\ell,j}:t_{2,\ell,j}:\cdots :t_{d_j,\ell,j}:1\right).
\end{equation}
These points $\alpha_{\ell,j}$  satisfy the following two conditions: 
\begin{itemize}
\item[(1)] given any \emph{distinct} pairs $(\ell_1,j_1),\dots, (\ell_u,j_u)$ (for some $u\in\N$) such that 
$$n_{\ell_1}^{(j_1)} = n_{\ell_2}^{(j_2)}=\cdots =n_{\ell_u}^{(j_u)},$$ 
(i.e., they all belong to the same part $\mathcal{P}_n$ in the partition of  $\mathcal{P}$), we have that  
\begin{equation}
\left\{t_{1, \ell_1, j_1}, \dots , t_{d_{j_1}, \ell_1, j_1}, t_{1, \ell_2, j_2}, \dots , t_{d_{j_2}, \ell_2, j_2},\dots,t_{1,\ell_u,j_u},\dots, t_{d_{j_u},\ell_u,j_u}\right\}    
\end{equation} 
is an algebraically independent set over $\Fpbar$ since the above $t_{k,\ell_u,j_u}$'s  are distinct elements of the transcendence basis for $K/\Fpbar$ due to the definition~\eqref{eq:t-l-j} along with the fact that the sums $S_{\ell_v,j_v}$ are all distinct for $v=1,\dots, u$ (and furthermore, if the pair $\left(j_{v_1},\ell_{v_1}\right)$ is smaller than the pair $\left(j_{v_2}, \ell_{v_2}\right)$, then $S_{\ell_{v_2},j_{v_2}}\ge S_{\ell_{v_1},j_{v_1}}+d_{j_{v_1}}$ due to definition~\eqref{eq:S-l-j}). 
\item[(2)] For any \emph{given} $1\le j\le r$ and any \emph{distinct} points $$\alpha_{\ell_1,j},\dots, \alpha_{\ell_u,j}\text{ for some }u\ge 1,$$
we have that these points are linearly independent over $\End(C_j)$. Indeed, since these points are distinct (which is equivalent, due to equation~\eqref{eq:alphaCoord 2},  with the fact that the sums $S_{\ell_v,j}$ are distinct for $v=1,\dots, u$), we have that each $\alpha_{\ell_v,j}$ is the generic point of the simple semiabelian variety $C_j$ in a different algebraically closed subfield $K_{v,j}\subset K$. Furthermore, letting - without loss of generality - $S_{\ell_u,j}$ be the largest sum among the sums $S_{\ell_v,j}$ (for $v=1,\dots, u$), then we have that $K_{u,j}$ is not contained in the compositum of the fields $K_{v,j}$ for $1\le v<u$. So, any linear dependence relation between the points $\alpha_{\ell_v,j}$ of the form
\begin{equation}
\label{eq:linear-dependence}
\sum_{v=1}^u \psi_v(\alpha_{\ell_v,j})=0
\end{equation} 
for some $\psi_1,\dots, \psi_u\in \End(C_j)$ would force that $\psi_u=0$. Then repeating the same reasnoning to the remaining $(u-1)$ distinct points $\alpha_{\ell_v,u}$ (for $1\le v<u$) yields that indeed the only possibility for equation~\eqref{eq:linear-dependence} to hold is when each endomorphism $\psi_v$ is the trivial one.  
\end{itemize}
Moreover, since $C_j(K)\otimes_\Z \Q$ is an infinite dimensional $\Q$-vector space, while $\End(C_j)$ is a finite $\Z$-module and also, each $\Gamma_j$ is a finitely generated $\End(C_j)$-module, one can choose the elements $\{t_1, \dots, t_d\}$ so that the following condition is also satisfied: 
\begin{itemize}
    \item[(3)] for each $j=1,\dots, r$, the $\End(C_j)$-submodule spanned by the action of the elements of $\End(C_j)$ on $\alpha_{i,j}$ (for $1\le i\le s_j$)  has trivial intersection with $\Gamma_j$.   
\end{itemize}

We construct the point
\[
\vec{\alpha} := (\vec{\alpha_1}, \dots, \vec{\alpha_r})\in G(K),
\]
where for each $j=1,\dots, r$, 
\begin{equation}
\vec{\alpha}_j = (\underbrace{\alpha_{1,j}, \dots, \alpha_{1,j}}_\text{$i_1^{(j)}$ times}, \underbrace{\alpha_{2, j}, \dots, \alpha_{2, j}}_\text{$i_2^{(j)} - i_1^{(j)}$ times}, \dots, \underbrace{\alpha_{s_j, j}, \dots, \alpha_{s_j, j}}_\text{$i_{s_j}^{(j)} - i_{s_j - 1}^{(j)}$ times})\in C_j^{k_j}(K).    
\end{equation}
Then,  condition~(i) from conclusion~(A) in Proposition~\ref{prop:noUnity} is satisfied by our choice for $\vec{\alpha}\in G(K)$ (see property~(3) above).  Next we prove that also condition~(ii) in conclusion~(A) holds for the orbit of $\vec{\alpha}$ under $\Phi$, i.e., in particular, we prove that its orbit $\OO_\Phi(\vec{\alpha})$ is Zariski dense in $G$.

So, we let $S_0\subseteq \N$ be a set of positive density, and we will prove that 
\begin{equation}
\label{eq:Set}
\mathcal{T}_{\Phi,S_0,\vec{\alpha}}:=\left\{\Phi^n(\vec{\alpha})\colon n\in S_0\right\}\text{ is Zariski dense in $G$.}
\end{equation}  
Suppose for the sake of contradiction that this is not the case. Let $V$ be the Zariski closure of the set $\mathcal{T}:=\mathcal{T}_{\Phi,S_0,\vec{\alpha}}$ from \eqref{eq:Set}. We let $\Gamma\subset G$ be the finitely generated $\Z[F]$-module consisting of all elements of the form $\sigma(\alpha)$, where $\sigma$ is in $\End(G)$; clearly, $\mathcal{T}\subseteq \Gamma$. Note that $\Gamma$ is indeed finitely generated as $\End(G)$ is a finitely generated module over $\Z$. By Theorem~\ref{Moosa-Scanlon theorem} (see also Section~\ref{sec:M-S}), we know that $V \cap \Gamma$ is a union of finitely many $F$-sets (just as in equation~\eqref{eqn:S-sets}; see also equation~\eqref{eqn:S-sets 22} below). Because $\mathcal{T}$ is contained in finitely many sets of the form 
\begin{equation}
\label{eqn:S-sets 22}
U:=\vec{\gamma} + \Sigma(\vec{\eta}_1,\dots,\vec{\eta}_t; \delta_1,\dots,\delta_t) + H,    
\end{equation}
then there must exist a given set $U$ of the form \eqref{eqn:S-sets 22} for which the following subset of $\N_0$:  
$$S_1=\left\{n\in S_0\colon \Phi^n(\vec{\alpha})\in U\right\}$$
has positive density $d(S_1)$. Furthermore, we know that there exists some positive integer $m$ such that
\begin{equation}
\label{eq:div-hull 22}
m\cdot\gamma,m\cdot\eta_1,\dots,m\cdot\eta_t\in \Gamma,
\end{equation} 
while the $\delta_j$'s are positive integers, $H$ is a subgroup of $\Gamma$ and (as before), we have the set 
$$\Sigma(\vec{\eta}_1,\dots,\vec{\eta}_t; \delta_1,\dots,\delta_t):=\left\{
\sum_{j=1}^t F^{\delta_jn_j}\cdot \vec{\eta}_j\colon n_j\in\N_0\text{ for }j=1,\dots, t\right\},$$
where $F:G\lra G$ is the Frobenius endomorphism corresponding to the field $\Fq$.

The algebraic closure of $H$ must be an algebraic group $\mathcal{H}$ contained in the stabilizer of the variety $W$, which is the Zariski closure of $U$. Since $V$ is a proper subvariety and $W\subseteq V$, then $\mathcal{H}$ must also be a proper algebraic subgroup of $G$. So, there must exist vectors $\vec{\sigma_i} = (\sigma_1^{(i)}, \dots, \sigma_{k_i}^{(i)}) \in \End(C_i)^{k_i}$ (for each $i=1,\dots, r$), not all the vectors $\vec{\sigma}_i$ being trivial, such that  the following equation holds: given any point $\left(\vec{\epsilon}_1, \dots, \vec{\epsilon}_r\right)\in H$, we have that  
\begin{equation}
\label{eq:v kills H 22}
\left(\vec{\epsilon_i}\right)^{\vec{\sigma_i}} = 1\text{ for each } i=1,\dots, r. 
\end{equation}

Using equation \eqref{eq:div-hull 22} along with the fact that $\Gamma$ is the cyclic $\End(G)$-module generated by $\vec{\alpha}$, we get that 
$$m\cdot \gamma= \tau(\alpha) \text{ and }m \cdot \eta_i = \tau_i(\alpha) \text{ for each }i=1,\dots, t,$$
where $\tau,\tau_1, \dots, \tau_t\in \End(G)$ and so,
\begin{equation}
\label{eq:n-in}
m\cdot \Phi^n(\vec{\alpha})=  \Big(\tau\left(\vec{\alpha}\right)+\sum_{j=1}^t F^{\delta_jn_j}\left(\tau_j\left(\vec{\alpha}\right)\right)\Big) + \vec{\upsilon}_n
\end{equation}
for some nonnegative integers $n_j$ and some $\vec{\upsilon}_n\in H$. 

For each $i=1,\dots, r$, we let $\Phi_i:=\Phi_{\big|C_i^{k_i}}$, which induces an endomorphism of $C_i^{k_i}$. On the other hand, for $\tau$ and also for $\tau_j$ (for $1\le j\le t$), we let $\tau^{(i)}$, respectively $\tau_j^{(i)}$ represent the restriction $\tau_{\big|C_i^{k_i}}$, respectively $(\tau_j)_{\big|C_i^{k_i}}$ which induce endomorphisms of $C_i^{k_i}$ for each $i=1,\dots, r$. Finally, we use $F^{(i)}$ to denote the Frobenius action on $C_i^{k_i}$ for each $i=1,\dots, r$. Combining \eqref{eq:n-in} with \eqref{eq:v kills H 22} yields that for each $i=1,\dots, r$, we have: 
\begin{equation}
\label{eq:n-in-2}
(m\cdot\Phi_i^n(\vec{\alpha}_i))^{\vec{\sigma_i}}= \Big( \Big(\tau^{(i)}(\vec{\alpha}_i) +\sum_{j=1}^t \Big(F^{(i)}\Big)^{\delta_jn_j}(\tau_j^{(i)}(\vec{\alpha}_i))\Big) \Big)^{\vec{\sigma_i}}.
\end{equation}

On the other hand, according to our hypothesis from Proposition~\ref{prop:noUnity}, we know that $\Phi_i^n(\vec{\alpha}_i)= (\vec{\alpha}_i)^{Q_i^n}$ and so, for each $k=1,\dots,r$ we have: 
\begin{align}
Q_k^n = \bigoplus_{j =1}^{s_k}\begin{pmatrix}
F_{C_k}^{n\cdot n_j^{(k)}} &  \binom{n}{1}F_{C_k}^{(n-1)\cdot n_j^{(k)}}& \cdots & \binom{n}{i_j^{(k)} - i_{j-1}^{(k)} - 1}F_{C_k}^{\left(n - i_j^{(k)} + i_{j-1}^{(k)} + 1\right)\cdot n_1^{(k)}} \\
0 & F_{C_k}^{n\cdot n_j^{(k)}} & \cdots & \binom{n}{i_j^{(k)} - i_{j-1}^{(k)} - 2}F_{C_k}^{\left(n - i_j^{(k)} + i_{j-1}^{(k)} + 2\right)\cdot n_j^{(k)}}\\
\vdots & \vdots & \ddots & \vdots \\
0 & 0 & \cdots & F_{C_k}^{n\cdot n_j^{(k)}}
\end{pmatrix}
\end{align}
where (as before) we use the convention that $i_0^{(k)}=0$. 

Next we employ the following technical Lemma. 
\begin{lemma}
\label{lem:j-k}
Let $u\in\{1,\dots, r\}$. Suppose now that some coordinate $\sigma_v^{(u)}$ of $\vec{\sigma}_u$ is nonzero. Then it must be that $v= i_{\ell-1}^{(u)}+1$ for some $\ell=1,\dots, s_{u}$. 
\end{lemma}

\begin{proof}[Proof of Lemma~\ref{lem:j-k}.]
We argue by contradiction and therefore, assume there exists some coordinate $v\ne i_{\ell-1}^{(u)}+1$ (for each $\ell=1,\dots, s_u$) such that $\sigma_v^{(u)}\ne 0$.

We let $\ell\in\{1,\dots, s_u\}$ be the unique integer for which we have
\begin{equation}
\label{eq:in-between}
i_{\ell-1}^{(u)} <v\le i_{\ell}^{(u)},
\end{equation}

Using  condition~(2) regarding the linear independence over $\End(C_u)$ of the distinct points $\alpha_{j,u}$, we must have that the coefficient (seen as an element of $\End(C_u)$) of $\alpha_{\ell,u}$  on the right hand side and respectively, on the left hand side of equation~\eqref{eq:n-in-2} must be equal. This means that there must exist $c, b_1, \dots, b_r \in \End(C_u)$ and polynomials $P_1, \dots, P_{s_u}$ with coefficients in $\End^0(C_j)$ such that 
\begin{equation}
\label{eqn:FrobOrbit}
F_{C_u}^{n\cdot n_1^{(u)}}P_1(n) + \cdots + F_{C_u}^{n\cdot n_{s_u}^{(u)}}P_{s_u}(n) = c + \sum_{i = 1}^t b_iF_{C_u}^{\delta_in_i},    
\end{equation}
for every $n \in S$. In the left hand side of equation~\eqref{eqn:FrobOrbit}, we collect the terms corresponding to the same value $n_i^{(u)}$ (as we vary $i\in\{1,\dots, s_u\}$) and so, we obtain a new equation:
\begin{equation}
\label{eqn:FrobOrbit_2}
F_{C_u}^{n\cdot \gamma_1}R_1(n) + \cdots + F_{C_u}^{n\cdot \gamma_k}R_{k}(n) = c + \sum_{i = 1}^t b_iF_{C_u}^{\delta_in_i},    
\end{equation}
where $\gamma_1,\dots, \gamma_k$ (for some $k\in\N$) are all the distinct $n^{(u)}_i$ (as we vary $i\in\{1,\dots, s_u\}$), while $R_1(n),\dots, R_k(n)$ are polynomials with coefficients in $\End^0(C_u)$.

\begin{claim}
\label{claim:crucial}
There exists $w\in\{1,\dots, k\}$ such that the polynomial $R_w(n)$ is not constant. 
\end{claim}

\begin{proof}[Proof of Claim~\ref{claim:crucial}.]
First of all, since we assumed that the entry $\sigma_v^{(u)}$ in $\vec{\sigma}_u$ is nonzero and $v\ne i^{(u)}_{h-1}+1$ for $h=1,\dots, s_u$, then using our definition of $\ell$ as in~\eqref{eq:in-between}, we get that  the polynomial
\begin{equation}
\label{eq:poly-nonconstant}
P_{\ell}(n)\text{ is nonconstant.}
\end{equation}
Now, using conditions~(1)~and~(2) satisfied by the points $\alpha_{j,u}$, it means that whenever $\alpha_{j,u}=\alpha_{\ell,u}$, we must also have that
\begin{equation}
\label{eq:different n's}
n^{(u)}_{j}\ne n^{(u)}_{\ell}.
\end{equation}
Equation~\eqref{eq:different n's} yields that when we collect terms in equation~\eqref{eqn:FrobOrbit} and derive equation~\eqref{eqn:FrobOrbit_2}, for the unique $w\in\{1,\dots k\}$ for which $\gamma_w=n^{(u)}_{\ell}$, we actually have that $R_w(n)=P_\ell(n)$. Then equation~\eqref{eq:poly-nonconstant} provides the desired conclusion in Claim~\ref{claim:crucial}.
\end{proof}

Now, note that $\End^0(C_u)$ is a vector space over $\Q[F_{C_u}]$. So, considering a basis for the $\Q[F_{C_u}]$-vector space $\End^0(C_u)$ (using the same argument from the proof of Proposition~\ref{prop:unipotentMain4}, as employed before equation \eqref{eq:density form equ}), we may assume without loss of generality that $c, b_1, \dots,b_r$ and the coefficients of the polynomials $R_w$ (for $w=1,\dots,k$, as in equation~\eqref{eqn:FrobOrbit_2}) are all contained in $\Q[F_{C_u}]$ which is a (commutative) field and can be viewed as a subset of $\C$. This contradicts \cite[Theorem~1.2]{GOSS}, which provides a upper bound for all positive integers $n\le N$ for which there exist some $n_i\in\Z$ such that 
$$u_n=\sum_{i=1}^t d_i a^{n_i},$$
where $a,c_1,\dots, c_t\in\C^*$ and $\{u_n\}$ is a linear recurrence sequence whose characteristic roots are not all simple and equal to powers of $a$. Indeed, the upper bound from \cite[Theorem~1.2]{GOSS} is of the form $O\left(\log(N)^t\right)$, while our hypothesis is that the set of $n$ satisfying the equation~\eqref{eqn:FrobOrbit_2} for some $n_1,\dots, n_t\in\N_0$  would have positive density. This concludes our proof of Lemma~\ref{lem:j-k}.
\end{proof}

Therefore, Lemma~\ref{lem:j-k} yields that for any $\sigma$ that kills all the elements in $H$, all the coordinates of $\sigma$ other than  $\sigma_{1}^{(j)}, \sigma_{i_1^{(j)} + 1}^{(j)}, \dots,\sigma_{i_{s_{j}-1}^{(j)} + 1}^{(j)}$ (for $j=1,\dots, r$) must be zero. This implies that $V \cong (\prod_{i = 1}^{r}C_i^{k_i - s_i}) \times Z$, where $Z$ is a subvariety of $\prod_{i=1}^{r} C_i^{s_i}$ containing the elements
\[
(F^{n\cdot n_1^{(1)}}(\alpha_{1,1}), \dots, F^{n\cdot n_{s_1}^{(1)}}(\alpha_{s_1, 1}), \dots,F^{n\cdot n_1^{(r)}}(\alpha_{1,r}), \dots, F^{n\cdot n_{s_r}^{(r)}}(\alpha_{s_r, r})). 
\]
Furthermore, for each $j=1,\dots, r$, we consider the first $d_j$ coordinates in $\Pl^{N_i}$ of the points $F^{n\cdot n_\ell^{(j)}}(\alpha_{\ell,j})$ (for $1\le \ell\le s_j$), i.e., we let:
\begin{equation}
\label{eqn:orbitCoords}
T_{N, \ell, j} := \left(t_{1,\ell,j}^{q^{N}},t_{2,\ell,j}^{q^N},\dots,t_{d_j,\ell,j}^{q^{N}}\right)\in \A^{d_j}(K),
\end{equation}
where we recall the definition of $t_{k,\ell,j}$ from \eqref{eq:t-l-j}. 

Since the dimension of $\prod_{i=1}^{r} C_i^{s_i}$ is equal to $e:=\sum_{i=1}^r s_id_i$, and because we assumed that $Z$ is a proper subvariety of $\prod_{j=1}^r C_j^{s_j}$, then there must exist a nonzero polynomial $\mathcal{Q}$ with coefficients in $K$ that vanishes on 
\begin{equation}
\label{eqn:orbitForm}
\left(T_{n\cdot n_1^{(1)}, 1, 1}, \dots, T_{n\cdot n_{s_r}^{(r)}, s_r, r}\right)\in \A^D(K),    
\end{equation}
for every $n \in S$ (see also the definition of $T_{N,\ell,j}$ from \eqref{eqn:orbitCoords}). So, for each $j=1,\dots, r$ and for each $1\le i\le s_j$, we let $\vec{x}_{i,j}$ be a vector with $d_j$ entries in $K$; then $\mathcal{Q} \in K[\vec{x}_{1,1}, \dots, \vec{x}_{s_r, r}]$ is a nonzero polynomial given by 
\[
\mathcal{Q}(\vec{x}_{1,1}, \dots,\vec{x}_{s_r, r})= \sum_{\vec{v}_{1,1}, \dots, \vec{v}_{s_r, r}}c_{\vec{v}_{1,1},\dots,\vec{v}_{s_r,r}}\cdot \vec{x}_{1,1}^{\vec{v}_{1,1}}\cdot \cdots \cdot \vec{x}_{s_r, r}^{\vec{v}_{s_r, r}}, 
\]
where $c_{\vec{v}_{1,1},\dots,\vec{v}_{s_r,r}} \in K$ and $\vec{v}_{i, j} \in \Z^{d_{j}}$ for every $j=1,\dots,r$ and $i=1,\dots,s_j$. Next we let 
\[
\mathbf{T} = \left\{t_{d, i, j}: 1 \le j \le r, 1 \le i \le s_j, 1 \le d \le d_j \right\},
\]
where the elements of $\mathbf{T}$ are not counted with repetition (i.e., the cardinality of $\mathbf{T}$ may be less than $e$). Then we may assume without loss of generality that $c_{\vec{v}_{1,1},\dots,\vec{v}_{s_r,r}}$ are polynomials in $\Fpbar[\mathbf{T}]$ (because any algebraic relation between the points from~\eqref{eqn:orbitForm} must already occur over $\Fpbar[\mathbf{T}]$). 

Let $D$ be the maximum (total) degree of the polynomials $c_{\vec{v}_{1,1},\dots,\vec{v}_{s_r,r}}$. For any $\tilde{t} \in \mathbf{T}$ we let $\text{deg}_{\tilde{t}}\left(P\right)$ denote the degree of $\tilde{t}$ in $P \in \Fpbar[\mathbf{T}\mathbf{T}]$. 

Now, since the elements in $\mathbf{T}$ are all algebraically independent (according to our choice for $t_{k,i,j}$ satisfying conditions~(1)-(2) from above), then the fact that $\mathcal{Q}$ vanishes at $\left(T_{n\cdot n_1^{(1)}, 1, 1}, \dots, T_{n\cdot n_{s_r}^{(r)}, s_r, r}\right)$, means that for each $n\in S$, there exist distinct vectors $\vec{v}_{1,1},\dots,\vec{v}_{s_r,r}$ and ${\vec{v}'_{1,1},\dots,\vec{v}'_{s_r,r}}$ such that for each $t\in \mathbf{T}$, we have that 
\begin{equation}
\label{eq:same at n}
\deg_{\tilde{t}}\left( c_{\vec{v}_{1,1},\dots, \vec{v}_{s_r,r}}\cdot \prod_{\substack{1\le j\le r\\ 1\le i\le s_j}} T_{n\cdot n_i^{(j)},i,j}^{\vec{v}_{i,j}}\right) = \deg_{\tilde{t}}\left( c_{\vec{v}'_{1,1},\dots, \vec{v}'_{s_r,r}}\cdot \prod_{\substack{1\le j\le r\\ 1\le i\le s_j}} T_{n\cdot n_i^{(j)},i,j}^{\vec{v}'_{i,j}}\right).
\end{equation}
At the expense of replacing $S$ with an infinite subset (actually, even a subset of positive density), we may actually assume that equation~\eqref{eq:same at n} holds for all $n\in S$. 

Next, for each $j=1,\dots, r$ and each $i=1,\dots, s_j$ and for each $\tilde{t}\in\mathbf{T}$, we let $\vec{u}^{(\tilde{t})}_{i,j}\in \Z^{d_j}$ be a vector whose $k$-th entry is either equal to $1$ or to $0$, depending on whether $t_{k,i,j}=\tilde{t}$, or not. Also, we let $\vec{w}_{i,j}:=\vec{v}_{i,j}-\vec{v}'_{i,j}$ for each $j=1,\dots, r$ and each $i=1,\dots, s_r$. Then equation~\eqref{eq:same at n}, along with the fact that the degrees of the polynomials $c_{\vec{v}_{1,1},\dots, \vec{v}_{s_r,r}}$ and $c_{\vec{v}'_{1,1},\dots, \vec{v}'_{s_r,r}}$ are bounded by $D$, we get the following inequality for each $n\in S$ and for each $\tilde{t}\in\mathbf{T}$: 
\begin{equation}
\label{eqn:degIneq}
\left|\sum_{\substack{1\le j\le r\\ 1\le i\le s_j}} q^{n\cdot n_i^{(j)}}\cdot \left(\vec{u}^{(\tilde{t})}_{i,j}\cdot \vec{w}_{i,j}\right)\right|\le D.
\end{equation}
We let $\gamma^{(\tilde{t})}_{i,j}\in\Z$ be the dot product of the vectors $\vec{u}^{(\tilde{t})}_{i,j}\cdot \vec{w}_{i,j}$. So, the inequality~\eqref{eqn:degIneq} yields that 
\begin{equation}
\label{eqn:degIneq_2}
\left|\sum_{\substack{1\le j\le r\\ 1\le i\le s_j}} q^{n\cdot n^{(j)}_i}\cdot \gamma^{(\tilde{t})}_{i,j}\right|\le D.
\end{equation}
Since $(\vec{v}_{1,1},\dots,\vec{v}_{s_r,r}) \ne (\vec{v}'_{1,1},\dots,\vec{v}'_{s_r,r})$ there must exist some $1\le j \le r$ and some $1 \le i \le s_j$ such that $\vec{w}_{i, j} \ne \vec{0}$. So, there exists $1 \le d \le d_j$ such that the $d$-th coordinate of $\vec{w}_{i,j}$ is non-zero. If we let $\tilde{t} = t_{d, i, j}$, then we see that equation \eqref{eqn:degIneq_2} becomes 
\begin{equation}
\label{eqn:degIneq_3}
\left|\sum_{1\le e \le E} \omega_e q^{n\cdot \kappa_e} \right|\le D    
\end{equation}
where $\omega_1, \dots,\omega_E$ are non-zero integers and $\kappa_1, \dots, \kappa_{E}$ are distinct positive integers. Indeed, due to condition (1) we know that $\gamma_{i, j}^{(\tilde{t})}$ and $\gamma_{i', j'}^{(\tilde{t})}$ can be non-zero if and only if $n_{i}^{(j)} \ne n_{i'}^{(j')}$. But, due to the fact that $\kappa_1, \dots, \kappa_{E}$ are distinct positive integers it is clear that the left hand side of \eqref{eqn:degIneq_3} must go to infinity as $n$ approaches infinity which is a contradiction. Therefore, $\OO_{\Phi}(\vec{\alpha})$ is Zariski dense in $G$ which concludes our proof of Proposition~\ref{prop:noUnity}.
\end{proof}

%%%%%%%%%%%%%%%%%%%%%%%%%%%%%%%%%%%%%%%%%%%%%%%%%%%%%%%%%%%%%%%%%%%%%%%%%%%%%%%%
%%%%%%%%%%%%%%%%%%%%%%%%%%%%%%%%%%%%%%%%%%%%%%%%%%%%%%%%%%%%%%%%%%%%%%%%%%%%%%%%

\section{A split case}
\label{sec:split_case}

Before proving the main result of this Section (which is Theorem~\ref{thm:main3}), we start with a technical Lemma regarding simple semiabelian varieties; actually, Lemma~\ref{lemma:pEqns} could be formulated solely  using subrings of skew fields which are integral over their center, but since its natural setting is the case of endomorphisms of simple semiabelian varieties, we prefer to formulate our result in this context.

\begin{lemma}
\label{lemma:pEqns}
Let $D$ be some simple semiabelian variety defined over a finite field $\Fq$, let $F$ be the Frobenius endomorphism of $D$ corresponding to the field $\Fq$, let $N,r\in\N$, let $\vec{v}$ be an $N$-by-$1$ vector with entries in $\End(D)$,  let $\delta_1, \dots, \delta_r \in \N$, and let $A, B_1, \dots, B_r, C$ be $N$-by-$N$ matrices with entries in $\End(D)$ such that $A$ is invertible and moreover, it is an $NFP$ matrix (see Definition~\ref{def:NFP}). If there exists an infinite subset $S \subseteq \N$ with the property that for each $n\in S$, there exist $n_1,\dots, n_r\in\N_0$ such that  
\begin{align}
A^n\vec{v} = C\vec{v} + \sum_{i = 1}^r F^{n_i\delta_i}B_i\vec{v}, \label{eqn:pEqn}
\end{align}
then $\vec{v}$ must be the zero vector. Similarly, if there exists an infinite subset $S \subseteq \N$ with the property that for each $n\in S$, there exist $n_1,\dots, n_r\in\N_0$ such that  
\begin{align}
\vec{v}^T A^n = \vec{v}^TC + \sum_{i = 1}^r F^{n_i\delta_i}\vec{v}^T B_i, \label{eqn:pEqn2}
\end{align}
then $\vec{v}$ must be the zero vector. 
\end{lemma}
\begin{remark}
\label{rem:pEqns}
Note that equations \eqref{eqn:pEqn} and \eqref{eqn:pEqn2} are not equivalent since if $A$ and $B$ are two matrices over a non-commutative ring, then $(AB)^T$ is not necessarily equal to $B^TA^T$. Having said that, a strategy that proves the first part of Lemma \ref{lemma:pEqns}, also proves the second part.
\end{remark}
\begin{proof}
We let $F_D$ be the image of the Frobenius in the endomorphism ring of $D$; we embed $\Q[F_D]$ into $\C$.

As noted in Remark \ref{rem:pEqns}, the proof for the two parts is similar, so we will only prove the first part. Suppose that $v$ is non-zero and there is an infinite subset $S \subseteq \N$ with the property that for each $n \in S$, there exists $n_1, \dots,n_r \in \N_0$ such that equation \eqref{eqn:pEqn} holds. Letting $$P(\lambda) = \lambda^{L} + a_{L-1}\lambda^{L-1} + \cdots + a_1\lambda + a_0$$
be the minimal polynomial of $A$ over $\Q[F_D]$ we see that for every $n$
\begin{equation}
\label{eqn:linearRecurrences}
A^n = \sum_{\ell = 0}^{L-1} a^{(\ell)}_{n}\left(A^{\ell}\right),   
\end{equation}
where for every $\ell = 0,\dots,L-1$, the sequence  $\left\{a^{(\ell)}_{n}\right\}_{n\in\N}$ is a linear recurrence with elements in $\Q[F_D]$ whose characteristic polynomial has roots that are all multiplicatively independent with respect to $F_D$ (since the roots of the polynomial $P$ are all multiplicatively independent with respect to $F_D$).  Therefore, if we let $\vec{u}_\ell := A^\ell\vec{v}$ for every $\ell = 0,\dots, L-1$, then we must have 
\begin{equation}
    \sum_{\ell = 0}^{L-1} a^{(\ell)}_{n}\vec{u}_\ell = C\vec{v} + \sum_{i = 1}^r F_D^{n_i\delta_i}B_i\vec{v}  
    \label{eqn:F-eqn3}
\end{equation}
Now consider the finitely generated vector space over $\Q[F_D]$ generated by the coordinates of $\vec{u}_0, \dots, \vec{u}_{L-1},C\vec{v},B_1\vec{v},\dots,B_r\vec{v}$. Let $\lambda_1, \dots, \lambda_s$ be a basis for this vector space. Then, using equations \eqref{eqn:F-eqn3}, for every $1 \le j \le s$ we get $N$ equations of the form
\begin{equation}
\label{eqn:linearRecurrences2}
    \sum_{\ell=0}^{L-1} d_\ell a^{(\ell)}_n = \sum_{j=1}^k c_{j} F_D^{n_j},
\end{equation}
where $d_{\ell}$'s and $c_{j}$'s are all inside $\Q[F_D]$. It is clear that for some $1 \le j \le s$, one of its corresponding $N$ equations (which are of the form \eqref{eqn:linearRecurrences2}) must be non-trivial,  i.e. the left hand side of equation~\eqref{eqn:linearRecurrences2} is not identically equal to zero. This contradicts Laurent's theorem \cite{Laurent} as the roots of the characteristic polynomials of $a_n^{(\ell)}$ are all multiplicatively independent with respect to $F_D$ and so, a nontrivial equation of the form~\eqref{eqn:linearRecurrences2} cannot be satisfied by infinitely many positive integers $n$. This concludes our proof of Lemma~\ref{lemma:pEqns}.
\end{proof}

The following result is the last technical ingredient that we require in order to derive Theorem~\ref{thm:main4 0 2}. In particular, Theorem~\ref{thm:main3} is obtained from Proposition~\ref{prop:noUnity} in a somewhat similar fashion as Proposition~\ref{prop:split} was deduced from Proposition~\ref{prop:unipotentMain4}.

\begin{theorem}
\label{thm:main3}
Let $K$ be an algebraically closed field of positive transcendence degree over $\Fpbar$, let $G = G_1 \times G_2$ be a split semiabelian variety where 
\begin{equation}
G_1 = \prod_{i = 1}^{r} C_i^{k_i}, \quad G_2 = \prod_{i = 1}^{r}C_i^{k'_i},
\end{equation}
and $C_1, \dots, C_r$ are non-isogenous simple semiabelian varieties defined over some finite subfield $\Fq\subset K$. (Note that we are allowing $k_1, \dots, k_r, k'_1, \dots, k'_r$ to be equal to zero in which case $C_i^0$ represents the trivial group.) Suppose we have the next almost commutative diagram 
\begin{equation}
\label{eq:almost_diagram}
\begin{tikzcd}
G' \arrow[r, "\Psi"] \arrow[d, "g"'] & G' \arrow[d, "g"] \\
G_1 \times G_2 \arrow[r, "{(\Phi_1, \Phi_2)}"] & G_1 \times G_2,
\end{tikzcd}
\end{equation}
where $G'$ is a split semiabelian variety, $\Psi$ is a group endomorphism of $G'$ and $g: G \lra G'$ is an isogeny. Moreover, $\Phi_1$ is a dominant group endomorphism of $G_1$ corresponding to matrices $A_1, \dots, A_r$ where $A_j \in M_{k_j, k_j}(\End(C_j))$ and each $A_j$ is of the form 
\begin{equation}
\label{eqn:jordanForm}
J_{F^{n_{1}^{(j)}},i_1^{(j)}}\oplus J_{F^{n_{2}^{(j)}},i_2^{(j)}-i_1^{(j)}}\oplus \cdots \oplus J_{F^{n_{s_j}^{(j)}},i_{s_j}^{(j)}-i_{s_j-1}^{(j)}}.    
\end{equation}
Also, we assume that $\Phi_2$ is a finite-to-finite map from $G_2$ to $G_2$ corresponding to matrices $A'_1, \dots, A'_{r}$ where $A'_i \in M_{k'_i, k'_i}\left(\frac{1}{m}\End(C_i)\right)$ for some $m \in \N$. Assume the following conditions are met:
\begin{itemize}
\item[(1)] The matrices $A'_1, \dots, A'_{r}$ are all NFP matrices.   
\item[(2)] $n_i^{(j)} \ge 1$ for every $1 \le j \le r$ and $1 \le i \le s_j$. 
\item[(3)] There does not exist $1 \le j_1 \le \cdots \le j_{\ell} \le r$, and $i_1, \dots, i_{\ell}$ satisfying $1 \le i_{k} \le s_{j_k}$ for every $1 \le k \le \ell$, such that the pairs $(i_k, j_k)$ are distinct and 
\[
n_{i_1}^{(j_1)} = \dots = n_{i_\ell}^{(j_\ell)},
\]
and 
\[
\dim(C_{j_1}) + \cdots + \dim(C_{j_{\ell}}) \ge \trdeg_{\Fpbar}K + 1.
\]
\end{itemize}
Then, given any finitely generated submodules $\Gamma_1, \dots, \Gamma_r$ where $\Gamma_i \subset C_i(K)$ is an $\End(C_i)$-submodule for every $1 \le i \le r$, there exist $\vec{\alpha} = (\vec{\alpha}_1, \dots, \vec{\alpha}_r)\in G_1(K)$ and $\vec{\beta} = (\vec{\beta}_1, \dots, \vec{\beta}_r)\in G_2(K)$ where $\vec{\alpha}_i \in C_i^{k_i}$ and $\vec{\beta}_i \in C_i^{k'_i}$ for every $1 \le i \le r$, such that
\begin{itemize}
\item[(i)] for each $i=1,\dots, r$, the $\End(C_i)$-module  spanned by $$\alpha_i^{(1)},\dots,\alpha_i^{(k_i)}, \beta_i^{(1)},\dots,\beta_i^{(k'_i)}$$  (which are the coordinates of $\vec{\alpha}_i$ and $\vec{\beta}_i$) has trivial intersection with $\Gamma_i$; and
\item[(ii)] for any subset $S$ of positive integers with positive density and any orbit $\{x_n\}_{n\ge 0}$ of $(\vec{\alpha}, \vec{\beta})$ under $\Phi := (\Phi_1,\Phi_2)$,  we have that the subset $\{x_n: n \in S\}$ is Zariski dense in $G$.  
\end{itemize}
\end{theorem}

\begin{proof}
Let $\Gamma_1, \dots,\Gamma_r$ be finitely generated submodules where $\Gamma_i \subset C_i(K)$ is an $\End(C_i)$-submodule for every $1 \le i \le r$. We pick a starting point $\vec{x} := (\vec{\alpha}, \vec{\beta})$ for the action of $(\Phi_1, \Phi_2)$ on
$G_1 \times G_2$ of the following form:  
\begin{itemize}
\item We pick $\vec{\alpha}_i \in C_i^{k_i}(K)$ for every $1 \le i \le r$ such that $\left(\vec{\alpha}_i\right)_{1 \le i \le r}$ satisfies both conditions~(i)-(ii) from the  conclusion of Proposition~\ref{prop:noUnity} with respect to the finitely generated subgroups $\Gamma_1, \dots, \Gamma_r$ (note that due to conditions~(2)-(3) of theorem \ref{thm:main3} along with the fact that the eigenvalues of the Jordan blocks in the Jordan canonical form of $A_i$ are of the form \eqref{eqn:jordanForm}, statement (A) in Proposition~\ref{prop:noUnity} must hold); and 
\item For each $i=1,\dots, r$, $\vec{\beta}_i$ has its $k'_i$ coordinates linearly independent among themselves over $\End(C_i)$ and also, the $\End(C_i)$-submodule of $C_i(K)$ generated by the coordinates of $\vec{\beta}_i$ has trivial intersection with the $\End(C_i)$-submodule spanned by $\Gamma_i$ and the  
coordinates of $\vec{\alpha}_i$.
\item We let $\vec{\alpha} := \left(\vec{\alpha}_1, \dots, \vec{\alpha}_r\right)$ and $\vec{\beta}:= \left(\vec{\beta}_1, \dots, \vec{\beta}_r\right)$.
\end{itemize}
At the expense of replacing the integer $m$ from Theorem~\ref{thm:main3} (for which $A'_i\in M_{k'_i, k'_i}\left(\frac{1}{m}\End(C_i)\right)$ for each $i=1,\dots, r'$) by a multiple of it, then we can find a group homomorphism $\hat{g}: G \lra G'$ such that 
\begin{equation}
\label{eq:gg}
\hat{g} \circ g = [m]_{G'}, \quad g \circ \hat{g} = [m]_{G}.
\end{equation}
In particular, we also have
\begin{equation}
\label{eq:almost_diagram_2}
[m]_G\circ \Phi= g\circ \Psi\circ \hat{g}.
\end{equation}
We let 
\[
\OO = \left\{\left(g \circ \Psi^n \circ \hat{g}\right)\left(\vec{x}_0\right): n \ge 0 \right\}, 
\]
where $\vec{x}_0 \in G_1 \times G_2$ is chosen such that $[m]_G(\vec{x}_0) = \vec{x}$. If we let $P_{\Psi}(x) = x^L + b_{L-1}x^{L-1} + \cdots + b_0$ be the minimal polynomial of $\Psi$ over $\Z$ then 
\begin{align}
\Lambda := \left\{\left( \sum_{i=0}^{L-1} a_i\left(g \circ \Psi^i \circ \hat{g}\right)\left(\vec{x}_0\right)\right): a_0,\dots,a_{L-1} \in \Z[F]\right\} \label{eqn:GammaForm} 
\end{align}
is a finitely generated $\Z[F]$-module (since the Frobenius $F:G\lra G$ is integral over $\Z$ in $\End(G)$); furthermore, all the points in $\OO$ are contained in $\Lambda$. 

We let $S\subseteq\N_0$ be an arbitrary set with positive density; we will prove that the set 
$$\OO_S:=\left\{\left(g \circ \Psi^n \circ \hat{g}\right)\left(\vec{x}_0\right)\colon n\in S\right\}$$  
must be Zariski dense in $G$. If $\OO_S$ is not Zariski dense, then we let $V\subset G$ be its Zariski closure.  
Using Theorem~\ref{Moosa-Scanlon theorem}, there must exist a set of the form \eqref{eqn:S-sets} containing infinitely many elements of $\OO_S$ (see also Section~\ref{sec:M-S}). So, at the expense of replacing $S$ by a smaller subset that still has positive density (and thus replacing the set $\OO_S$ with its corresponding infinite subset), then there exists a set 
\begin{align}
\mathcal{F}:=\vec{\lambda} + \Sigma(\vec{\eta}_1,\dots,\vec{\eta}_s; \delta_1,\dots,\delta_s) + H,  \label{eqn:S-sets 5}  
\end{align}
containing $\OO_S$. Now, regarding the set $\mathcal{F}$ (see also Remark~\ref{rem:M-S}), there exists a positive integer $m_2$ such that   
\begin{equation}
\label{eq:div-hull 5}
m_2\cdot \vec{\lambda},m_2\cdot \vec{\eta}_1,\dots,m_2\cdot \vec{\eta}_s\in \Lambda,
\end{equation} 
while the $\delta_j$'s are positive integers and $H$ is a subgroup of $\Lambda$. 
Since we assumed that $\OO_S$ is not Zariski dense in $G$, then $V$ is a proper subvariety of $G$ and in particular, the Zariski closure of $H$ must be a proper algebraic subgroup of $G$; so, there must exist an endomorphsim $\sigma:G_1 \times G_2\lra G_1\times G_2$ such that $\sigma(\vec{\epsilon}) = 0$ for every $\vec{\epsilon} \in H$. 

If we let 
\[
\vec{\rho} := \left((g \circ \hat{g}), \left(g \circ \Psi \circ \hat{g}\right), \dots, \left(g \circ \Psi^{L-1} \circ \hat{g}\right)\right),
\]
and
\[
\vec{\rho}(\vec{x}) := \left((g \circ \hat{g})(\vec{x}), \left(g \circ \Psi \circ \hat{g}\right)(\vec{x}), \dots, \left(g \circ \Psi^{L-1} \circ \hat{g}\right)(\vec{x})\right),
\]
for every $x \in G$, then using equation \eqref{eq:div-hull 5}, for every $i = 1,\dots,s$ there exist vectors $\vec{u}_0, \vec{u}_1, \dots, \vec{u}_s \in \Z[F]^{L}$ such that 
\begin{align}
&m_2\cdot\vec{\eta}_i = \vec{u_i} \cdot \rho(\vec{x}_0) \text{ for each $i=1,\dots, s$ and  } m_2\cdot \vec{\lambda} = \vec{u_0} \cdot \vec{\rho}\left(\vec{x}_0\right).  \notag    
\end{align}  
So, for each $n \in S$, using that $\left(g \circ \Psi^n \circ \hat{g}\right)\left(\vec{x}_0\right) \in \mathcal{F}$, for every $j=1\dots,r$ we must have some some non-negative integers $n_i$ (for $i=1,\dots, r$, where the $n_i$'s depend on $n$) such that 
\begin{align}
\sigma\left(m_2\left(g \circ \Psi^n \circ \hat{g}\right)\left(\vec{x}_0\right)\right) &= \sigma\left(\left( \vec{u}_0\cdot \vec{\rho} + \sum_{i=1}^s F^{n_i\delta_i}(\vec{u}_i\cdot\vec{\rho})\right)\left(\vec{x}_0\right)\right).
\label{eq:last equality}
\end{align}
Due to the way the coordinates of $\vec{x}$ are chosen, we know that there does not exist a nontrivial endomorphism of $G_1 \times G_2$ that vanishes at $\vec{x}$. Considering the fact that $[m]_{G}(\vec{x}_0) = \vec{x}$, we also deduce that there does not exist any non-trivial endomorphism of $G_1 \times G_2$ that vanishes at $\vec{x}_0$. Therefore, we must have   the following equality taken place inside $\End(G)$: 
\begin{equation}
\label{eq:equalEndomorphisms}
\sigma\left(m_2\left(g \circ \Psi^n \circ \hat{g}\right)\right) = \sigma\left(\left( \vec{u}_0\cdot \vec{\rho} + \sum_{i=1}^s F^{n_i\delta_i}(\vec{u}_i\cdot\vec{\rho})\right)\right)
\end{equation}
for every $n \in S$. Let the group endomorphism $\sigma$ correspond to some matrix $P$ whose rows are of the form $\vec{v}_1 \oplus \vec{v}_2$ where $\vec{v}_1 \in \prod_{i=1}^r \End(C_i)^{k_i}$ and $\vec{v}_2 \in \prod_{i=1}^r \End(C_i)^{k'_i}$. Using \eqref{eq:gg}~and~\eqref{eq:almost_diagram_2}, we know that $m_2(\left(g \circ \Psi^n \circ \hat{g}\right))$ corresponds to matrices which are similar to $m_2m\cdot A_1^n, \dots, m_2m\cdot A_r^n, m_2m\cdot (A'_1)^n, \dots, m_2m\cdot (A'_r)^n$. Now, since each $A'_i$ is an NFP matrix, Lemma~\ref{lemma:pEqns} and equation \eqref{eq:equalEndomorphisms} yield that for every row $\vec{v}_1 \oplus \vec{v}_2$ of $P$ we must have $\vec{v}_2 = 0$. This clearly holds for every $\sigma \in \End(G_1 \times G_2)$ that kills the elements of $H$. Therefore, $\overline{H}$ is an algebraic group of the form $\overline{H}_1 \oplus G_2$ for some algebraic subgroup $\overline{H}_1 \subseteq G_1$.

So, the Zariski closure $W$ of the set $\mathcal{F}$ (which is itself contained in the Zariski closure of the set $\OO_S$) must be of the form $W_1 \oplus G_2$ for some subvariety $W_1\subseteq G_1$ because $ \vec{1}_{G_1} \oplus G_2$ is contained in the stabilizer of $W$. However, $W_1$ contains all the points   $\Phi_1^n(\vec{\alpha})$ for $n \in S$. Then using the fact that $S$ is an infinite subset of $\N_0$ along with Proposition~\ref{prop:noUnity}, we conclude that $W_1$ must be the entire $G_1$. So, actually $W$ must be the entire $G_1\oplus G_2=G$, which proves that for any $S \subset \N_0$ with positive density, the corresponding set $\OO_S$ is Zariski dense in $G$. 

Now, take any orbit $\{y_n\}_{n\ge 0}\subset G(K)$ of $\vec{x}$. Arguing exactly as in the proof of Proposition \ref{prop:split} we see that since $\OO_S$ is Zariski dense in $G$ then $\{y_n: n \in S\}$ must also be Zariski dense in $G$.

This concludes our proof of Theorem~\ref{thm:main3}.
\end{proof}

%%%%%%%%%%%%%%%%%%%%%%%%%%%%%%%%%%%%%%%%%%%%%%%%%%%%%%%%%%%%%%%%%%%%%%%%%%%%%%%
%%%%%%%%%%%%%%%%%%%%%%%%%%%%%%%%%%%%%%%%%%%%%%%%%%%%%%%%%%%%%%%%%%%%%%%%%%%%%%%

\section{Conclusion for our proof of Theorem~\ref{thm:main4 0}}
\label{sec:conclusion}

In this Section we prove Theorem~\ref{thm:main4 0 2}, which in turn provides the desired conclusion in Theorem~\ref{thm:main4 0}. 

\begin{proof}[Proof of Theorem \ref{thm:main4 0 2}]
Let us assume that conditions (B) and (C) do not hold. We will show that there must exist $\alpha \in G(K)$ with a Zariski dense orbit. 

Let $\vec{\beta} = (\vec{\beta}_1, \dots, \vec{\beta}_{r})$ where $\vec{\beta}_i \in C_i^{k_{0, j}}$ for every $j = 1,\dots,r$. At the expense of replacing $\Phi$ by a conjugate of the form $\tau_{\vec{\gamma}}^{-1}\circ \Phi\circ \tau_{\vec{\gamma}}$, where $\tau_{\vec{\gamma}}$ is a suitable translation map corresponding to a vector, we may assume that $\vec{\beta}_j:=(1,\dots, 1,\beta_{i_{0, 1}^{(j)}}^{(j)},1,\dots, 1,\beta_{i_{0, s_j}^{(j)}}^{(j)})\in C_j^{k_{0, j}}(K)$ for every $1 \le j \le r$. Since condition~(B) is not met then for every $j=1,\dots,r$ the $\beta_{i_{0, k}^{(j)}}^{(j)}$'s must be linearly independent over $\End(C_j)$. Indeed, suppose there exist $\sigma_1,\dots,\sigma_{s_j} \in \End(C_j)$ not all equal to zero such that 
\[
\sigma_1\left(\beta_{i_{0, 1}^{(j)}}^{(j)}\right) + \cdots + \sigma_{s_j}\left(\beta_{i_{0, s_j}^{(j)}}^{(j)}\right) = 0.
\]
If we let $\Pi: G \lra C_j^{k_{0, j}}$ be the natural projection map onto $C_j^{k_{0, j}}$, let $f_1: C_j^{k_{0, j}} \lra C_j$ be the map given by 
\[
\left(x_1,\dots, x_{i_{0, 1}^{(j)}-1},x_{i_{0, 1}^{(j)}},x_{i_{0, 1}^{(j)}+1},\dots, x_{i_{0, s_j}^{(j)}-1},x_{i_{0, s_j}^{(j)}}\right) \mapsto \sigma_1\left(x_{i_{0, 1}^{(j)}}\right) + \cdots + \sigma_{s_j}\left(x_{i_{0, s_j}^{(j)}}\right)  
\]
and $f_2:C_j \dra \Pl^1$ be a non-constant rational map it is clear that $\Phi_1$ is invariant under $f := f_2 \circ f_1 \circ \Pi$ which contradicts our initial assumption that condition (B) does not hold. So, we must have that for every $j=1,\dots,r$ the $\beta_{i_{0, k}^{(j)}}^{(j)}$'s are linearly independent over $\End(C_j)$.

Let   
$$\vec{\alpha}_{Q_j}:=\left(\gamma_1^{(j)},\dots, \gamma_{i_{0, 1}^{(j)}-1}^{(j)},1,\gamma_{i_{0, 1}^{(j)}+1}^{(j)},\dots, \gamma_{i_{0, 2}^{(j)}-1}^{(j)},1,\gamma_{i_{0, 2}^{(j)}+1}^{(j)},\dots, \gamma_{i_{0, s_j}^{(j)}-1}^{(j)},1\right)\in C_j^{k_{0,j}}(K)$$
where the $\gamma_k^{(j)}$'s are linearly independent over $\End(C_j)$ and also linearly independent with respect to the $\beta_{i_{0, k}^{(j)}}^{(j)}$'s and let 
\[
\vec{\alpha}_1 := \left(\vec{\alpha}_{Q_1}, \dots, \vec{\alpha}_{Q_r}\right).
\] 

For every $j=1,\dots,r$ let $\Gamma_j$ be the $\End(C_j)$-submodule of $C_j(K)$ generated by the $\beta_{i_{0, k}^{(j)}}^{(j)}$'s and all the $\gamma_k^{(j)}$'s. Since condition (C) is not met, $(\varphi_1, \varphi_2)$ satisfies the hypotheses of Theorem~\ref{thm:main3}. So, we can find $\vec{\alpha}_2 \in (G_1 \times G_2)(K)$ whose coordinates satisfy conditions~(i)-(ii) from the conclusion of Theorem~\ref{thm:main3} with respect to $\Gamma_1, \dots,\Gamma_r$. In particular, this means that the coordinates of $\vec{\alpha}_1$ (along with the $\beta_{i_{0, k}}^{(j)}$'s and the $\gamma_k^{(j)}$'s) satisfy the hypotheses of Proposition~\ref{prop:split}. Hence, the orbit of $\left(\vec{\alpha}_1\oplus\vec{\alpha}_2\right) \in (G_0 \times G_1 \times G_2)(K)$ under $(\Phi, \varphi_1, \varphi_2)$ must be Zariski dense in $G_0 \times G_1 \times G_2$, as claimed. This concludes our proof of Theorem~\ref{thm:main4 0 2}.
\end{proof}

As shown in Section~\ref{sec:reduction}, Theorem~\ref{thm:main4 0} is a  consequence of Theorem~\ref{thm:main4 0 2}.

%%%%%%%%%%%%%%%%%%%%%%%%%%%%%%%%%%%%%%%%%%%%%%%%%%%%%%%%%%%%%%%%%%%%%%%%%%%%%%%
%%%%%%%%%%%%%%%%%%%%%%%%%%%%%%%%%%%%%%%%%%%%%%%%%%%%%%%%%%%%%%%%%%%%%%%%%%%%%%%

\end{document}